\documentclass[11pt,oneside,reqno]{amsart}
 \usepackage{amsmath,amsthm,amscd,amssymb,
 }

\usepackage[left,pagewise,mathlines]{lineno}

\usepackage[margin=1in]{geometry}

\usepackage{booktabs}
\usepackage[toc]{appendix}

\providecommand{\keywords}[1]{\textbf{\textit{Keywords---}} #1}

\usepackage{lipsum}

\usepackage{bbm}

\usepackage{latexsym}

\usepackage{amsmath, amsthm, amssymb, latexsym,epsfig,amsthm, enumerate,multicol,wasysym}
\usepackage{bm}
\usepackage{mathtools}
\usepackage[dvipsnames]{xcolor}

\usepackage{graphicx}
\usepackage{nccrules}
\usepackage{xfrac}

 \usepackage{hyperref}
 \hypersetup{
    colorlinks,
    citecolor=red,
    filecolor=red,
    linkcolor=red,
    urlcolor=black
}
\usepackage{textcomp}
\usepackage{tikz}
\usepackage{comment}
\usepackage{enumitem}

\numberwithin{equation}{section}

\theoremstyle{plain}
\newtheorem{theorem}{Theorem}[section]
\newtheorem{lemma}[theorem]{Lemma}
\newtheorem{corollary}[theorem]{Corollary}
\newtheorem{proposition}[theorem]{Proposition}

\newtheorem{problem}[theorem]{Problem}

\theoremstyle{definition}
\newtheorem{definition}[theorem]{Definition}

\newtheorem{example}[theorem]{Example}

\usepackage{mwe}
\usepackage{wrapfig}

\theoremstyle{remark}
\newtheorem{remark}[theorem]{Remark}

\numberwithin{equation}{section}

\newcommand{\ffi}{\varphi}

\newcommand{\abs}[1]{\left\lvert#1\right\rvert}

\def\R{\mathbb{R}}
\def\C{\mathbb{C}}

\usepackage[normalem]{ulem}

\usepackage{latexsym}

\makeatletter
\renewcommand\@makefntext[1]{%
  \noindent
  \makebox[1em][r]{\@makefnmark}#1}
\makeatother

\usepackage{ulem}

\def\dst{\displaystyle}

\newcommand{\scal}[1]{{\left\langle{#1}\right\rangle}}

\newcommand{\T}{\mathbb T}

\DeclareMathOperator{\dist}{dist}

 \DeclareMathOperator{\diag}{diag}

\newcommand{\trivial}{\sim_{\mathbb{T}}}

\newcommand{\partner}{e^{-itH_Q}}

\def\R{\mathbb{R}}

\makeatletter
\renewcommand{\l@section}{\@tocline{1}{0pt}{2em}{}{}}
\renewcommand{\l@subsection}{\@tocline{2}{0em}{3.2em}{}{}}
\renewcommand{\l@subsubsection}{\@tocline{3}{0em}{4.0em}{}{}}
\makeatother

\author{Philippe Jaming}
\address{Univ. Bordeaux, CNRS, Bordeaux INP, IMB, UMR 5251, F-33400 Talence, France}
\email{philippe.jaming@math.u-bordeaux.fr}

\author{Azita Mayeli}
\address{Department of Mathematics, CUNY, The Graduate Center, NY}
\email{amayeli@gc.cuny.edu}

\begin{document}

\title[Dynamical phase retrieval]{Dynamical phase retrieval for Schr\"odinger evolution on finite graphs}

\thanks{P.J. was supported in part by the French National
Research Agency (ANR) under contract number ANR-24-CE40-5470.}

\thanks{A.M. was supported in part by the National Science Foundation grant DMS-2453769, and an AMS-Simons Research Enhancement Grant.}

\date{\today}


\begin{abstract}
We study dynamical phase retrieval for Schr\"odinger evolutions on finite connected
graphs. Let
\[
H_Q=\Delta_G+Q
\]
be a graph Schr\"odinger operator with a real diagonal potential. We investigate when
phaseless data obtained from the associated Schr\"odinger evolution
\[
|e^{-itH_Q}u_0(j)|,
\qquad
0\leq t\leq T,\ j\in V,
\]
determines the initial state $u_0\in\C^V$ up to a global phase.

We give a uniqueness criterion in terms of the eigenvalues and eigenvectors of 
$H_Q$. The assumptions are a $B_2$ condition on
the spectrum, meaning that the sums $\lambda_j+\lambda_k$ determine the
unordered pair $\{j,k\}$, invertibility of the squared-eigenvector matrix
$\bigl(\phi_k(j)^2\bigr)_{j,k}$
and an overlap condition on the supports of pairs of eigenvectors. Under these
hypotheses, the phaseless Schr\"odinger data determine every initial state
uniquely, modulo global phase.

We then show that the criterion is both realized and generic. Every finite
connected graph admits an explicit real diagonal potential for which the
criterion holds. Moreover, for every finite connected graph, dynamical phase
retrieval holds for Lebesgue-almost every real potential $Q\in\R^V$ and every $T>0$.
We also give several obstructions to uniqueness.
\end{abstract}

\maketitle


\section{Introduction}
The recovery of a vector from the moduli of structured linear measurements is closely related to the spectral geometry of matrices. 
In this paper, we study a dynamical version of this problem generated by a real symmetric matrix $H$: given the coordinatewise magnitudes
$$
\bigl|e^{-itH}u_0(j)\bigr|,\qquad t\in[0,T],\quad j=1,\ldots,n,
$$
we ask whether the initial vector $u_0\in\mathbb C^n$ is determined up to a global phase. 
We prove uniqueness under conditions on the eigenvalues, the zero patterns of the eigenvectors, and the matrix of their squared coordinates that lead to uniqueness. Although much of our analysis extends to arbitrary real symmetric matrices, we focus here on the case where $H$ is a diagonal perturbation of the combinatorial Laplacian of a finite graph, and hence is a graph Schrödinger operator. This setting provides both a natural framework and a physical interpretation for the underlying matrix-theoretic question, while extending the well-studied Pauli problem in quantum mechanics.

\subsection{Background and motivation}

The family of phase retrieval problems, namely the reconstruction of a function from its modulus and {\it a priori} information on the function, is an essential problem in applied sciences. It appears in many fields, including signal and image processing, crystallography, and quantum physics. We refer to \cite{AllainAslanCoeneEtAl2025,GrohsKoppensteinerRathmair2020,KlibanovSacksTikhonravov1995,Mixon2015PhaseTransitions,MP,ShechtmanEldarCohenEtAl2015} for some surveys on this problem.

The precise problem we are interested in here stems from a fundamental question originally asked by W. Pauli
in his foundational treatise {\it General Principles of Quantum Mechanics}~\cite{Pa}: is the state of a quantum particle uniquely determined by its position and momentum probability distributions?
Mathematically, this question is usually reformulated in terms of the Fourier transform defined
for $\psi\in L^1(\R^d)\cap L^2(\R^d)$ as
$$
\widehat{\psi}(\xi)=\int_{\R^d}\psi(x)e^{-i\scal{x,\xi}}\,\mbox{d}x
$$
and extended to $L^2(\R^d)$ in the usual way. Pauli's question can then be reformulated as follows:

\begin{problem}[Pauli's Problem]
Given $\psi\in L^2(\R^d)$, what is the set of all $\ffi\in L^2(\R^d)$ such that
$|\ffi(x)|=|\psi(x)|$ for all $x\in\R^d$ and $|\widehat{\ffi}(\xi)|=|\widehat{\psi}(\xi)|$
for all $\xi\in\R^d$\,?

We then say that $\ffi,\psi$ are {\em Pauli partners}.
\end{problem}

The problem can also be stated for the discrete Fourier transform defined for $\psi\in\C^n$
by
$$
\mathcal{F}_n[\psi](j)=\sum_{k=1}^n \psi(k)e^{-2i\pi(j-1)(k-1)/n}\quad,\ j=1,\ldots,n.
$$
This leads to a finite-dimensional formulation.

Of course, if $\ffi,\psi$ are equal up to a constant phase factor,  that is, if there is
a $c\in\C$ with $|c|=1$ such that $\ffi=c\psi$, then they are Pauli partners.

\smallskip

Throughout this article, we write $\T=\{z\in\C\,:\ |z|=1\}$ and,
when $f,g$ are two functions (resp. vectors in $\C^n$), we say 

\medskip

\begin{center}
$f,g$ are {\em trivial partners}, denoted
$f\trivial g$, if there is a $c\in\T$ such that $f=cg$;
\end{center}

\medskip 

\noindent that is, $f,g$ are equal up to a global phase.

Pauli thus asked whether all partners should be trivial. Counterexamples were quickly constructed
by Bargmann using symmetry properties of the Fourier transform. Despite further investigation
\cite{BelousovIsmagilov2008Pauli,Co,CH,Is,Jaming1999Radar,Janssen1992Zak}
that led to proving uniqueness in some classes and provided more counterexamples,  
a satisfactory solution to the problem is still out of reach.

Later, this lack of uniqueness in Pauli's problem led R. Wright
to conjecture the existence of a third unitary operator $T$ on $L^2(\R^d)$,
possibly with a physical meaning, such that
$|\ffi|=|\psi|$, $|\widehat{\ffi}|=|\widehat{\psi}|$ and $|T\ffi|=|T\psi|$
implies $\ffi\trivial\psi$. This question (as well as its attribution to Wright) can be found
in H. Reichenbach's book \cite{Re} and is still open in the continuous setting (see below for the finite-dimensional setting).

Motivated by this question,   the first author    
reformulated Pauli's problem in 
terms of the free Schr\"odinger equation in \cite{JamingAcha}:
$$
\begin{cases}
i\partial_t u(t,x)=-\Delta_x u(t,x),
& t>0,\ x\in\mathbb R^d,\\[2mm]
u(0,x)=u_0(x),
& x\in\mathbb R^d.
\end{cases}
$$
When $u_0\in L^2(\R^d)\cap L^1(\R^d)$, the solution is then given by
$$
u(t,x)=\frac{e^{i d\pi/4}}{(4\pi t)^{d/2}}e^{i|x|^2/4t}
\int_{\mathbb{R}^d}e^{-i|y|^2/4t}u_0(y)e^{-i\scal{y,x/2t}}\,\mathrm{d}y
=\frac{e^{-i d\pi/4}}{(4\pi t)^{d/2}}e^{i|x|^2/4t}\widehat{\psi_t}(x/2t)
$$
where $\psi_t(y)=u_0(y)e^{i|y|^2/4t}$. Note that $|\psi_t|=|u_0|$
so that Pauli's problem for $\psi_t$ reformulates as follows:

\begin{center}
    {\sl Do $|u(0,\cdot)|$
and $|u(t,\cdot)|$ uniquely determine $u_0$, up to a constant phase factor?}
\end{center}

As the answer is negative, and in view of Wright's conjecture
it becomes natural to ask whether measuring the modulus of the solution at more times leads to uniqueness up to a constant phase factor.
Further, it is then natural to ask the same question for the Schr\"odinger
operator with a reasonable potential $Q$.
We then consider $H_Q=-\Delta+Q$ and the associated Schr\"odinger equation
$$
\begin{cases}
i\partial_t u(t,x)=H_Q u(t,x),
& t>0,\ x\in\mathbb R^d,\\[2mm]
u(0,x)=u_0(x),
& x\in\mathbb R^d.
\end{cases}
$$
For $u_0\in L^2(\R^d)$, we write the solution in terms of the Schr\"odinger propagator: $u(t,\cdot)=e^{-itH_Q}u_0$
which makes sense for reasonable potentials. We then ask the following:

\begin{problem}[Schr\"odinger--Pauli Problem]
Let \(Q\) be a potential on \(\mathbb R^d\), let
\[
H_Q=-\Delta+Q,
\]
and let \(\tau\subset \mathbb R^+\) be a set of observation times.

Given \(u_0\in L^2(\mathbb R^d)\), determine the set of all
\(v_0\in L^2(\mathbb R^d)\) such that
\[
\bigl|e^{-itH_Q}v_0\bigr|
=
\bigl|e^{-itH_Q}u_0\bigr|
\quad\text{a.e. on }\mathbb R^d,\qquad t\in\tau .
\]
We call such \(v_0\) the \((H_Q,\tau)\)-Schr\"odinger--Pauli partners of \(u_0\).
\end{problem}

Ideally, one would like to prove that, for every \(u_0\), or at least for
a large class of initial data, all \((H_Q,\tau)\)-Schr\"odinger--Pauli partners
of \(u_0\) are trivial, that is, coincide with \(u_0\) up to the global-phase 
ambiguities of the problem.
Apart from the free one-dimensional case \(Q=0\), \(d=1\), treated in 
\cite{JamingAcha}, this question remains largely open.

The advantage of the formulation of the Schr\"odinger--Pauli Problem is that it makes sense not only on $\R^d$ but in every geometric structure
in which the Schr\"odinger equation makes sense, such as Riemannian manifolds, finite or infinite graphs,
quantum graphs, etc. The aim of this paper is to start a study 
of this question by focusing on finite graphs.

\subsection{Problem setting}

Recall that a finite graph of $n$ vertices is a pair $G=(V,E)$ where $V=\{1,\ldots,n\}$
is the set of vertices and $E\subset V\times V$ is the set of edges. For $j,k\in V$, we will say that they are adjacent and write
$j\sim k$ if $\{j,k\}\in E$. The combinatorial Laplacian on $G$ is then defined for a function
$u\,:V\to\C$ by
$$
\Delta_G u(j)=\sum_{k\sim j}\bigl(u(j)-u(k)\bigr).
$$
Next, we identify $Q\in\R^{V}=\R^n$ with the operator $\C^{V}\to\C^{V}$ obtained by entrywise multiplication, namely 
$Qu(j)=Q(j)u(j)$. We call   $Q$  a  {\em potential} and write $H_Q=\Delta_G+Q$. For a fixed $T>0$, the Schr\"odinger equation on $G$ with potential $Q$ and  initial datum 
$u_0\in\C^{V}$ is then the equation
$$
\left\{\begin{matrix}
i\partial_tu(t,j)&=&\Delta_G u(t,j)+Q(j)u(t,j)&t\in[0,T],\ j\in V,\\
u(0,j)&=&u_0(j)&j\in V.
\end{matrix}\right.
$$
By the Cauchy-Lipschitz theorem, for every $u_0\in\C^{V}$, there is a unique global solution
that we denote by $e^{-itH_Q}u_0$.

We would like to emphasize that, although one could allow complex-valued potentials, this would lead to a
non-self-adjoint Schr\"odinger operator and hence to a non-unitary evolution.
Since our motivation is the conservative Schr\"odinger dynamics on a graph, we
restrict ourselves to {\em real-valued potentials}.

The Schr\"odinger equation on a finite graph describes the unitary propagation of a complex-valued wave function 
whose geometry is encoded by the edges of the graph. Spectral properties of the graph Laplacian, such as eigenvalues 
and eigenvectors, determine the long-time behavior of the solutions. This finite-dimensional model is useful both as 
a discrete analogue of the continuous Schr\"odinger equation (recall that the Laplacian on the
path $\{1,\ldots,n\}$ with edges $j\sim j+1$ is the natural discretization of $-\partial_x^2$)
and as a model for quantum dynamics on networks.
When $\|u_0\|_2=1$, then $|u(t,j)|^2$ represents the probability of finding a particle at time $t$ at the vertex $j$.
The question we are addressing here is whether the corresponding probability distribution
over a set of times $\tau$ uniquely determines the wave function $u_0$.
Here  we restrict attention to the case
$\tau=[0,T]$ and ask the following question:

\begin{problem}[Schr\"odinger-Pauli Problem on finite graphs]\label{prob:schrodinger-Pauli}
Let $T>0$, $G=(V,E)$ be a finite graph, $\Delta_G$ its combinatorial Laplacian,
$Q$ a potential on $V$ and $H_Q=\Delta_G+Q$. Two vectors $u_0,v_0\in\C^{V}$
will be called $\partner$-partners if $|e^{-itH_Q}u_0(j)|=|e^{-itH_Q}v_0(j)|$ for every $t\in[0,T]$ and every $j\in V$.  Is every pair of  $\partner$-partners a pair of trivial partners? 
In this case, we say that $H_Q$ does {\em Schr\"odinger phase retrieval}.
\end{problem}

The main goal of this paper is to answer Problem \ref{prob:schrodinger-Pauli} by identifying simple spectral conditions under which the Schr\"odinger evolution on a finite graph performs phase retrieval. We will introduce three conditions, named $B_2$, \textup{(I)}, and \textup{(S)},
that imply uniqueness up to a global phase for every initial state. Those are
related to the eigenvalue distribution, the matrix of the square coordinates of the eigenvectors
and their zero pattern. We will then show that these conditions are generically satisfied after adding a diagonal potential to $\Delta_G$ on every connected graph. Notice that, unlike in the $L^2(\mathbb{R}^d)$-setting mentioned above, we do not speak here of ($\partner,\tau$)-partners. Indeed, one of the first outcomes of our analysis is that the size of the time interval $\tau$ has no influence on the results and that the continuous-time measurements may even be reduced to a finite set of observation times.

\subsection{Measurement operator and frame interpretation}

It is useful to reformulate the dynamical phase retrieval problem in terms of a measurement operator. For a fixed graph $G=(V,E)$, potential $Q$, and time interval $[0,T]$, define the map
\[
\mathcal{M}_{Q,T} : \C^V \to \R_+^{V \times [0,T]}, \qquad
\mathcal{M}_{Q,T}(u_0) = \big( |e^{-itH_Q} u_0(j)| \big)_{t \in [0,T],\, j \in V}.
\]
The Schr\"odinger phase retrieval problem then asks whether $\mathcal{M}_{Q,T}$ is injective modulo the natural $\T$-action, that is,
\[
\mathcal{M}_{Q,T}(u_0) = \mathcal{M}_{Q,T}(v_0)
\quad \Longrightarrow \quad
v_0 \sim_{\T} u_0.
\]

For each fixed $t$ and $j$, the quantity $e^{-itH_Q}u_0(j)$ can be written as an inner product
\[
e^{-itH_Q}u_0(j) = \langle u_0, f_{t,j} \rangle,
\]
where $f_{t,j} = e^{itH_Q}\delta_j$. 
Consequently, since \(e^{-itH_Q}\) is unitary, $(f_{t,j})_{t \in [0,T],\, j \in V}$ forms a continuous tight  frame for $\C^V$ with frame constant $T$. Moreover, the map $\mathcal{M}_{Q,T}$ is the corresponding phaseless measurement map for the frame coefficients.

Equivalently, the dynamical phase retrieval problem may be seen as a phase retrieval problem for a structured continuous frame generated by the Schr\"odinger evolution. The conditions introduced below, namely \(B_2\), \textup{(I)}, and \textup{(S)},
ensure that this frame does phase retrieval.

\smallskip

\noindent{\it Remark.} Although the frame $(f_{t,j})_{t \in [0,T],\, j \in V}$  is continuous in $t$, Lemma \ref{lem:linind}  below shows that it can be reduced to a finite set of times without loss of information.

\smallskip

At this point, we outline some of the related literature on phase retrieval from frame coefficients. Recall that, in finite dimensions, a frame is simply a spanning family. The corresponding problem may be formulated as follows.

\begin{problem}[Frame phase retrieval problem]
For which frames $(f_j)_{j=1}^N$ of $\mathbb C^n$ does
$$
|\langle u_0,f_j\rangle|=|\langle v_0,f_j\rangle|,
\qquad j=1,\ldots,N,
$$
imply that $u_0\trivial v_0$?
In this case, one says that the frame does phase retrieval.
\end{problem}

There is now an extensive mathematical literature on this problem, initiated by the seminal work of Balan \emph{et al.} \cite{BalanBodmannCasazzaEdidin2009,BCE}, and further stimulated by the development of new algorithms based on convex relaxations
\cite{CandesEldarStrohmerVoroninski2013,CandesStrohmerVoroninski2013,WaldspurgerDAspremontMallat2015}. These complement the earlier algorithms based on alternating projections, for which convergence is still not fully understood; see, for instance, \cite{BauschkeCombettesLuke2002,Fienup1982}.

Several works concern phase retrieval for structured frames, such as Gabor frames or coded diffraction patterns ({\it see e.g.} \cite{BF,GKK} and references therein), and the present paper may be viewed, to some extent, as belonging to this line of work. 
Here, we would like to highlight the work of Aldroubi {\it et al.} \cite{AKT}
and Beinert and Hasannasab \cite{BH} who also investigate dynamical phase retrieval from frames.
However, their measurement scheme is different since they are looking for vectors $(f_k)$
and {\em real} matrices $A$ such that, for some sequence of {\em integers} $\{t_j\}$, $|\scal{u_0,A^{t_j}f_k}|=
|\scal{v_0,A^{t_j}f_k}|$ implies $u_0\trivial v_0$. Taking $f_k$ to be the standard basis,
our situation would correspond to $A=e^{iH_Q}$, but the times $\{t_j\}$
need not be integers, and $A$ is not a real matrix.
The common feature is that the spectral properties of $A$
play a key role in both findings.

Another important question is motivated by Wright's conjecture and by the search for frames that do phase retrieval with minimal cardinality. A natural way of constructing frames with good properties is to take unions of orthonormal bases of
$\mathbb{C}^n$. Among these, a standard class of candidates is provided by mutually unbiased bases: recall that two orthonormal bases $(e_j)_{j=1}^n$ and $(f_j)_{j=1}^n$ are said to be mutually unbiased if
$|\langle e_j,f_k\rangle|=n^{-1/2}$,
 for $1\leq j,k\leq n$.
Wright's conjecture would then assert that the union of three orthonormal bases does phase retrieval. This was disproved by Moroz and Perelomov \cite{MP}, whereas Balan \emph{et al.} \cite{BalanBodmannCasazzaEdidin2009} proved that the union of
$n+1$ mutually unbiased orthonormal bases does phase retrieval. Such a family of bases is, however, known to exist only when $n$ is a prime power. Heinosaari \emph{et al.} \cite{HMW} proved the stronger lower bound that the minimal number of vectors in a frame doing phase retrieval in $\mathbb{C}^n$ is at least $3n+\alpha_n$, with $\alpha_n\to+\infty$ as
$n\to+\infty$. In unpublished work, Mondragon and Voroninski \cite{MV} showed that the union of four generic orthonormal bases does phase retrieval; a later published paper \cite{Go} proves the corresponding result for five bases.

Our aim here is different. Although one can discretize the \emph{continuous frame} considered here
and thereby obtain an ordinary finite frame, this would lead to a frame with far more than $4n$ elements.
The point is rather that the frame considered here has a natural physical meaning, in the spirit of the second aspect of Wright's conjecture. 

\subsection{Main results}

We now describe the main results of the paper. We fix a finite graph $G=(V,E)$,
and we assume it is connected. This condition is easily seen to be necessary for any $H_Q=\Delta_G+Q$, $Q\in\R^V$
to do phase retrieval.
Next $H_Q$ is a {\em real} self-adjoint operator\footnote{A substantial part of this paper can be extended
directly
to arbitrary self-adjoint operators, thus also including Schr\"odinger operators on undirected weighted graphs.} on $\mathbb C^V$  so that the spectral theorem
provides us with an orthonormal basis of eigenvectors $(\phi_k)_{k=1}^n$ with
$(\lambda_k)_{k=1}^n\in\R^n$ the corresponding (real) eigenvalues. Note that the eigenvectors
may (and will) be chosen to be real vectors.  Let \(u_0,v_0\in \mathbb C^V\). We decompose them in this basis 
\[
u_0=\sum_{k=1}^na_k\phi_k\quad\mbox{and}\quad
v_0=\sum_{k=1}^nb_k\phi_k
\]
and solve the Schr\"odinger equation
\[
e^{-itH_Q}u_0=\sum_{k=1}^na_ke^{-i\lambda_kt}\phi_k
\quad\mbox{and}\quad
e^{-itH_Q}v_0=\sum_{k=1}^nb_ke^{-i\lambda_kt}\phi_k.
\]
The question we are then asking is the following: let $a=(a_k)_{k=1}^n$, $b=(b_k)_{k=1}^n\in\C^n$ satisfy
\begin{equation}
    \label{eq:thequestion}
\abs{\sum_{k=1}^na_ke^{-i\lambda_kt}\phi_k(j)}^2
=\abs{\sum_{k=1}^nb_ke^{-i\lambda_kt}\phi_k(j)}^2
\quad \mbox{for all }t\in[0,T]\mbox{ and }j\in\{1,\ldots,n\}. 
\end{equation}
Does this imply that $a\trivial b$?

In this formulation, it becomes obvious that the spectral decomposition of $H_Q$ plays a central role.
Let us now introduce the three properties for $H_Q$ that are involved:

-- we will say that $H_Q$ has Condition $B_2$ if the set of eigenvalues is a $B_2$-set,  namely, 
$\lambda_j+\lambda_k=\lambda_{j'}+\lambda_{k'}$ if and only if $\{j,k\}=\{j',k'\}$ (in particular, the eigenvalues
are simple). This property allows us to reduce \eqref{eq:thequestion} to a set of algebraic
equations on $a_k,b_k$, namely
\begin{equation}
    \label{eq:thereduction}
    \left\{
    \begin{matrix}\dst
\sum_{k=1}^n|a_k|^2\phi_k(j)^2&=&\dst\sum_{k=1}^n|b_k|^2\phi_k(j)^2&&j=1,\ldots,n\\
a_k\overline{a_{\ell}}\phi_k(j)\phi_\ell(j)&=&b_k\overline{b_{\ell}}\phi_k(j)\phi_\ell(j)&\quad&k\not=\ell,\ j=1,\ldots,n
    \end{matrix}\right..
\end{equation}

-- The second condition, denoted by $\textup{(I)}$, is that the squared-eigenvector matrix
\[
\bigl(\phi_k(j)^2\bigr)_{j,k=1}^n
\]
is invertible. With this property, the first line of \eqref{eq:thereduction} reduces to
\[
|a_k|^2=|b_k|^2,
\qquad k=1,\ldots,n.
\]

-- The last condition we need, denoted by $\textup{(S)}$, is that the supports of every  pair of  eigenvectors $\phi_k,\phi_\ell$
intersect, {\it i.e.},  that there exists $j$ such that $\phi_k(j)\phi_\ell(j)\not=0$.
Of course, if every eigenvector has full support, then Condition $\textup{(S)}$ holds.

We say that $H_Q$ satisfies Conditions $B_2$--\textup{(I)}--\textup{(S)} if it satisfies
each of the three conditions. Our main result is then the following:

\begin{theorem}\label{thm:mainth}
    Let $G=(V,E)$ be a connected graph and $Q\in\R^V$. Assume that $H_Q$
    satisfies Conditions $B_2$--\textup{(I)}--\textup{(S)}. Then $H_Q$ does Schr\"odinger phase retrieval.
\end{theorem}

We will show in Proposition \ref{prop:b2andsimpliesS} that for small graphs of at most 5 vertices,
Conditions $B_2$ and \textup{(I)} imply Condition \textup{(S)}. This is not true for graphs
of at least $6$ vertices.

Next, we highlight the role of the potential $Q$. First, the unperturbed Laplacians
never satisfy Condition \textup{(I)} ({\it see} Section \ref{sec:genI}). Further, the Laplacians of many classical graphs, including path graphs, 
cycle graphs,  complete graphs,
complete bipartite graphs, etc., typically fail Condition $B_2$ (there are even often multiple eigenvalues due to the symmetries present in such graphs); {\it see,  e.g., } 
\cite{brouwer2011spectra}. Thus, the potential $Q$ plays a key role.

We prove two results in this direction. First, we construct explicitly, for
every connected graph $G=(V,E)$, a potential $Q\in\mathbb R^V$ such that
$H_Q$ satisfies Conditions $B_2$--\textup{(I)}--\textup{(S)}. This is
Proposition~\ref{prop:explicit-good-potential-connected}. Second, we show
that this property is generic and it holds for Lebesgue-almost every potential: if $G$ is connected, then for almost every
$Q\in\mathbb R^V$, the operator $H_Q$ satisfies Conditions
$B_2$--\textup{(I)}--\textup{(S)}. Consequently, Schr\"odinger
phase retrieval holds generically: 

\begin{theorem}\label{thm:generic}
     Let $G=(V,E)$ be a connected graph. Then, for almost every $Q\in\R^V$, $H_Q$ does Schr\"odinger phase retrieval.
\end{theorem}

The result is actually a bit stronger, as the set of $Q$'s to be eventually excluded is a proper algebraic set.

Finally, we will also construct graphs and potentials that fail to do phase retrieval.
For instance, when $G$ is not connected or when $H_Q$ does not satisfy Condition \textup{(S)} then
$H_Q$ fails to do phase retrieval. The same is true if $H_Q$ has eigenvalues with multiplicity,
thus failing to satisfy   Condition $B_2$. We provide examples where this happens.
We also provide examples where Condition  \textup{(I)} is not satisfied, and $H_Q$ fails to do phase retrieval.

\begin{remark}
Theorem~\ref{thm:mainth} is purely spectral. Its proof uses only that $H_Q$ is a
real symmetric matrix satisfying Conditions $B_2$--\textup{(I)}--\textup{(S)}.
Thus, the same conclusion holds for any real symmetric matrix $H$ on $\mathbb C^n$,
with measurements $|e^{-itH}u(j)|$ in the fixed coordinate basis.
 We would, however,  lose a part of the physical interpretation
of the result.

Theorem \ref{thm:generic},  on the other hand, uses the connectedness of the graph
to establish the genericity of Condition  \textup{(S)}. It can be extended to weighted graphs.
\end{remark}

\subsection{Organization of the paper}

The paper is organized as follows. In Section~\ref{sec:spectral-criterion} we prove the main spectral
criterion, Theorem~\ref{thm:mainth}. We also prove the classification of
partners under the $B_2$ and \textup{(I)} assumptions,
Proposition~\ref{prop:classification}, and derive the almost-everywhere
uniqueness criterion, Corollary~\ref{cor:ae-phase-retrieval}. We then record
a finite-time sampling result, Proposition~\ref{prop:finite-time-sampling},
and a frame-measurement variant of the same criterion.

Section~ \ref{sec:diagonal-potentials} is devoted to diagonal potentials. We then prove that every connected graph admits an
explicit diagonal potential for which  Conditions $B_2$--\textup{(I)}--\textup{(S)} hold, Proposition~\ref{prop:explicit-good-potential-connected}.
The rest of the section proves the genericity of these conditions under
diagonal perturbations: Proposition~\ref{prop:generic-B2} for the Condition $B_2$, 
Proposition~\ref{prop:generic-I} for Condition \textup{(I)}, and
Theorem~\ref{thm:generic-S} for Condition \textup{(S)}. Together, these results
give the generic uniqueness theorem, Theorem~\ref{thm:generic}.

In Section~\ref{sec:obstructions} we discuss the main obstructions to uniqueness. We treat
disconnected graphs, multiple eigenvalues and eigenvectors
with the same pointwise modulus. Some of those obstructions occur on product graphs.

The paper concludes in Section~\ref{sec:conclusion}
 with several remarks and open questions,
including the unperturbed case $Q=0$ and algorithmic aspects of the problem.

\subsection*{Main contributions.}
The main contributions of this paper may be summarized as follows.

\begin{itemize}
\item We formulate Schr\"odinger phase retrieval on finite graphs as a phase retrieval problem for a structured continuous frame generated by the Schr\"odinger evolution.

\item We prove that continuous-time observations may be replaced by finitely many sampling times without loss of information.

\item We identify three simple spectral conditions, namely $B_2$, \textup{(I)}, and \textup{(S)}, and prove that together they imply uniqueness up to a global phase for every initial state.

\item We prove that these conditions are generically satisfied after diagonal perturbations of the graph Laplacian and construct explicit diagonal potentials on every connected graph for which Schr\"odinger phase retrieval holds.

\item We identify graph-theoretic obstructions to phase retrieval and show how the support-intersection property of eigenvectors governs uniqueness.

\end{itemize}

\section{The spectral criterion}\label{sec:spectral-criterion}
\subsection{Proof of Theorem \ref{thm:mainth}}

We prove the theorem in two steps. 

\smallskip

\noindent{\bf Step 1.} {\em Reduction to \eqref{eq:thereduction} with Conditions $B_2$ and $\textup{(I)}$.}

As explained in the introduction, we fix a finite graph $G=(V,E)$ with $V=\{1,\ldots,n\}$, and a potential $Q\in \R^V$.
Then $H_Q=\Delta_G+Q$ is self-adjoint and has a real orthonormal basis of eigenvectors
$(\phi_j)_{j=1}^n$ with $$H_Q\phi_j=\lambda_j\phi_j, \qquad \lambda_j\in \mathbb R.$$  
Take $u_0,v_0\in\C^V$ and decompose them with respect to this basis: 
$$
u_0=\sum_{k=1}^na_k\phi_k\quad\mbox{and}\quad v_0=\sum_{k=1}^nb_k\phi_k.
$$
Then 
$$
e^{-itH_Q}u_0=\sum_{k=1}^na_ke^{-i\lambda_k t}\phi_k\quad\mbox{and}\quad e^{-itH_Q}v_0=\sum_{k=1}^nb_ke^{-i\lambda_k t}\phi_k.
$$
It follows that the equality $|e^{-itH_Q}u_0(j)|^2=|e^{-itH_Q}v_0(j)|^2$ is equivalent to
\[
\sum_{k,\ell=1}^n
\big(a_k\overline{a_\ell}-b_k\overline{b_\ell}\big)
e^{-i(\lambda_k-\lambda_\ell)t}
\phi_k(j)\phi_\ell(j)=0
\]
for all $t\in[0,T]$ and all $j\in V$.


We will now use the following simple, well-known lemma:

\begin{lemma}\label{lem:linind}
Let $\mu_1,\ldots,\mu_N\in\mathbb C$ be distinct. Then the functions
$e^{i\mu_1 t},\ldots,e^{i\mu_N t}$ are linearly independent on every interval
$[0,T]$, $T>0$. Equivalently,
\[
F(t):=\sum_{r=1}^N c_r e^{i\mu_r t}\equiv 0 \ \text{on } [0,T]
\quad\Longrightarrow\quad
c_r=0 \ \text{for all } r.
\]
\end{lemma}

\begin{proof}[Sketch of proof]
It is enough to notice that differentiating successively $N-1$ times $F$ and evaluating at $0$, we obtain a
nonsingular Vandermonde system, hence $c_r=0$ for all $r$.    
\end{proof}

Now, if $H_Q$ satisfies $B_2$, 
by Lemma \ref{lem:linind}, we obtain that the following two statements are equivalent:
\begin{enumerate}
\renewcommand{\theenumi}{\roman{enumi}}
    \item $|e^{-itH_Q}u_0(j)|=|e^{-itH_Q}v_0(j)|$ for all $j$ and $t\in[0,T]$

    \item the coefficients $a=(a_k)_{k=1}^n$, $b=(b_k)_{k=1}^n$ of $u_0$ and $v_0$ in the eigenbasis of $H_Q$ satisfy
    the two equations
\begin{equation}
    \label{eq:firstcond}
\sum_{k=1}^n|a_k|^2\phi_k(j)^2=\sum_{k=1}^n|b_k|^2\phi_k(j)^2 \qquad j\in\{1,\ldots,n\},\quad \text{and}
\end{equation}
\begin{equation}
    \label{eq:secondcond}
a_k\overline{a_\ell}\phi_k(j)\phi_\ell(j)=b_k\overline{b_\ell}\phi_k(j)\phi_\ell(j) \quad j\in\{1,\ldots,n\}
\quad k\not=\ell.
\end{equation}
\end{enumerate}
Finally, if $H_Q$ satisfies Condition $\textup{(I)}$, then \eqref{eq:firstcond} is equivalent to
\begin{equation}
    \label{eq:thirdcond}
|a_k|=|b_k| \qquad k\in\{1,\ldots,n\}.
\end{equation}

\begin{remark}
    Note that \eqref{eq:secondcond}-\eqref{eq:thirdcond} do not depend on $T$.
    Consequently, if the equality of phaseless measurements,
    $|e^{-itH_Q}u_0(j)|=|e^{-itH_Q}v_0(j)|$ for all $j$, holds on one non-degenerate interval $t\in [0,T]$,
    then it holds on every non-degenerate time-interval.
\end{remark}

\smallskip

\noindent{\bf Step 2.} {\em Propagation of the phases.}

Define a graph $\Gamma_Q$ on the index set $\{1,\ldots,n\}$ by declaring
$k\sim \ell$ if there exists $j\in V$ such that
\[
\phi_k(j)\phi_\ell(j)\neq 0.
\]
Condition \textup{(S)} says precisely that $\Gamma_Q$ is the complete graph.

We prove the following slightly stronger statement.

\begin{proposition}[Classification of partners under $B_2$ and \textup{(I)}]\label{prop:classification} 
Assume that $H_Q$ satisfies Conditions $B_2$ and \textup{(I)}.
Let
\[
u_0=\sum_{k=1}^n a_k\phi_k,
\qquad
v_0=\sum_{k=1}^n b_k\phi_k.
\]
Let $A=\{k:a_k\neq0\}$.
Then $u_0$ and $v_0$ are $\partner$-partners if and only if
\[
\{k:b_k\neq0\}=A
\]
and, for every connected component $K$ of the induced graph $\Gamma_Q(A)$,
there exists $c_K\in\mathbb T$ such that
\[
b_k=c_Ka_k,
\qquad k\in K.
\]
\end{proposition}

\begin{proof} We have already shown that $u_0$ and $v_0$ are $\partner$-partners
if and only if the sequences $a=(a_k)_{k=1}^n$ and $b=(b_k)_{k=1}^n$ satisfy \eqref{eq:secondcond} and \eqref{eq:thirdcond}.
This last one implies that $a$ and $b$ have the same support $A$.
Let $K$ be a connected component of $\Gamma_Q(A)$, and choose $k_0\in K$.
Since $a_{k_0}$ and $b_{k_0}$ are nonzero and have the same modulus, there is
$c_K\in\mathbb T$ such that
\[
b_{k_0}=c_Ka_{k_0}.
\]
If $k\sim \ell$ in $\Gamma_Q(A)$, then for some $j\in V$,
$\phi_k(j)\phi_\ell(j)\neq0.$
Hence \eqref{eq:secondcond} reduces to
\begin{equation}
    \label{eq:secondcondbis} 
a_k\overline{a_\ell}
=
b_k\overline{b_\ell}.
\end{equation}
Thus if $b_\ell=c_Ka_\ell$ then $b_k=c_Ka_k$. In other words, the 
identity \eqref{eq:secondcondbis} propagates the phase factor $c_K$ along every path of $\Gamma_Q(A)$. Therefore,
along any path in $K$, one obtains
\[
b_k=c_Ka_k.
\]
Thus the phase factor is constant on each connected component of $\Gamma_Q(A)$.

Conversely, suppose that for every connected component $K$ of $\Gamma_Q(A)$
there exists $c_K\in\mathbb T$ such that
\[
b_k=c_Ka_k,
\qquad k\in K.
\]
Then $|a_k|=|b_k|$ for every $k$. If $k,\ell$ lie in the same connected
component, then
\[
b_k\overline{b_\ell}
=
c_Ka_k\overline{c_Ka_\ell}
=
a_k\overline{a_\ell}.
\]
If $k,\ell$ lie in different connected components of $\Gamma_Q(A)$, then
\[
\phi_k(j)\phi_\ell(j)=0
\qquad\text{for every } j\in V,
\]
so the mixed equations are trivial. The cases where one of the indices is
outside $A$ are also trivial. Hence the phaseless measurements agree.
\end{proof}

We now finish the proof of Theorem~\ref{thm:mainth}. By Condition \textup{(S)},
the graph $\Gamma_Q$ is complete. Hence, every induced subgraph $\Gamma_Q(A)$ is
connected whenever $A\neq\emptyset$. Proposition \ref{prop:classification}
therefore gives a single $c\in\mathbb T$ such that
\[
b_k=c a_k,
\qquad k=1,\ldots,n.
\]
Thus
\[
u_0=c v_0.
\]
Hence $u_0\trivial v_0$, and $H_Q$ does Schr\"odinger phase retrieval.


The classification also gives the generic-in-the-state version of the problem.

\begin{corollary}[Almost-everywhere phase retrieval]\label{cor:ae-phase-retrieval}
Assume that $H_Q$ satisfies the Conditions $B_2$  and \textup{(I)}.
Then the following are equivalent:
\begin{enumerate}
\item for Lebesgue-a.e. $u_0\in\C^V$, the only $\partner$-partners of $u_0$
are trivial;
\item the graph $\Gamma_Q$ is connected.
\end{enumerate}
Moreover, uniqueness for every $u_0\in\C^V$ holds if and only if $\Gamma_Q$ is
complete, equivalently if Condition \textup{(S)} holds.
\end{corollary}
\begin{proof}
Write $ u_0=\sum_{k=1}^n a_k\phi_k$. 
For Lebesgue-a.e. $u_0$, all coefficients $a_k$ are nonzero; the exceptional set
is the finite union of the coordinate hyperplanes $\{a_k=0\}$. Thus, for almost
every $u_0$,
\[
A=\{1,\ldots,n\}.
\]
By Proposition~\ref{prop:classification}, the partners of $u_0$ are obtained by
choosing one unimodular constant on each connected component of $\Gamma_Q$.
Hence, almost-everywhere uniqueness is equivalent to the connectedness of
$\Gamma_Q$.

For uniqueness for every $u_0$, it is necessary that every nonempty induced graph
$\Gamma_Q(A)$  be connected. This is equivalent to $\Gamma_Q$ being complete,
which is Condition \textup{(S)}.
\end{proof}

\subsection{Examples}

We now give three examples. First, we give a full solution for two-vertex graphs
that will show the role of the potential. We will then give two examples
of graphs and potentials such that $H_Q$ satisfies Condition \textup{(S)}, does phase retrieval,
though only one of Conditions \textup{(I)} and $B_2$ holds.

\begin{example}[The two-vertex graph]\label{ex:2vertex}
We now consider the simplest nontrivial connected graph, namely the graph \(P_2\)
consisting of two vertices joined by one edge. This example has two purposes.
First, it gives an explicit Schr\"odinger operator satisfying the three hypotheses
\(B_2\),  \textup{(I)}, and   \textup{(S)} of Theorem~\ref{thm:mainth}. Second, it shows that
Condition  \textup{(I)} fails in the symmetric case, and that this failure leads to
non-uniqueness.

Let \(Q(1)=q_1\) and \(Q(2)=q_2\). Then
\[
\Delta_{P_2}
=
\begin{pmatrix}
1&-1\\
-1&1
\end{pmatrix},
\qquad
H_Q=\Delta_{P_2}+Q
=
\begin{pmatrix}
\alpha&-1\\
-1&\beta
\end{pmatrix},
\]
where $\alpha=1+q_1$, $\beta=1+q_2$.
The eigenvalues are
\[
\lambda_\pm
=
\frac{\alpha+\beta\pm\sqrt{4+(\alpha-\beta)^2}}{2}.
\]
The two eigenvalues are distinct,  and furthermore  $2\lambda_-<\lambda_-+\lambda_+<2\lambda_+$,
thus, Condition $B_2$ holds.

The corresponding real eigenvectors may be chosen as $\phi_\pm=(1,\alpha-\lambda_\pm)$.
These eigenvectors are not normalized, but this is harmless for checking
Conditions  \textup{(I)} and  \textup{(S)}, since normalizing only multiplies the columns of
\([\phi_k(j)^2]\) by nonzero constants. Condition  \textup{(S)} is satisfied since $\phi_+(1)=\phi_-(1)=1\neq 0$.

It remains to be noticed that
$$
\det[\phi_k(j)^2]=(\alpha-\lambda_-)^2-(\alpha-\lambda_+)^2
=\sqrt{4+(\alpha-\beta)^2}\,(\alpha-\beta).
$$
Thus Condition  \textup{(I)} holds if and only if $\alpha\not=\beta$, that is, if $q_1\not=q_2$.

On the other hand, it is not difficult to solve the system \eqref{eq:thereduction} if $q_1=q_2$.
Indeed, if we write $a=(a_1,a_2)=(|a_1|e^{i\ffi_1},|a_2|e^{i\ffi_2})$
then either $b\trivial a$ or $b\trivial (|a_2|e^{i\ffi_1},|a_1|e^{i\ffi_2})$.

We conclude that {\em on $P_2$, $H_Q$ does Schr\"odinger phase retrieval if and only if $Q$ is not
a constant potential.}
\end{example}
Our second example is the following:
\begin{example}[A graph satisfying \textup{(I)} and \textup{(S)}, but not
$B_2$, for which $H_Q$ does phase retrieval] 
Consider the star graph \(G=K_{1,3}\), with vertex \(1\) as its center, and let
\(Q=\operatorname{diag}(-4,-2,-1,1)\). Therefore, the corresponding Schrödinger operator is
\[
H_Q=\Delta_G+Q=
\begin{pmatrix}
-1&-1&-1&-1\\
-1&-1&0&0\\
-1&0&0&0\\
-1&0&0&2
\end{pmatrix}.
\]
Its characteristic polynomial is
\[
\det(\lambda I-H_Q)=\lambda^4-6\lambda^2+2,
\]
and hence, if we set
\[
\alpha=\sqrt{3+\sqrt7},
\qquad
\beta=\sqrt{3-\sqrt7},
\]
then the eigenvalues are simple and given by
\[
(\lambda_1,\lambda_2,\lambda_3,\lambda_4)=(-\alpha,-\beta,\beta,\alpha).
\]
 Then the complete difference table is
 \[
 \bigl(\lambda_j-\lambda_k\bigr)_{1\leq j,k\leq 4}
 =
 \begin{pmatrix}
0 & -(\alpha-\beta) & -(\alpha+\beta) & -2\alpha\\
\alpha-\beta & 0 & -2\beta & -(\alpha+\beta)\\
\alpha+\beta & 2 \beta & 0 & -(\alpha-\beta)\\
2\alpha & \alpha+\beta& \alpha-\beta & 0
\end{pmatrix}.
\]
In particular,
\[
\lambda_1+\lambda_4=0=\lambda_2+\lambda_3,
\]
while $\{1,4\}\neq\{2,3\}$. Hence, the spectrum does not satisfy Condition
$B_2$. 

 For each eigenvalue $\lambda_k$, an associated eigenvector may be chosen as
$$
\phi_k=\begin{pmatrix}
1& -\dfrac{1}{1+\lambda_k}&-\dfrac{1}{\lambda_k}&\dfrac{1}{2-\lambda_k}\end{pmatrix}.
$$
None of the eigenvalues equals $-1$, $0$, or $2$, so each $\phi_k$ has full
support. Hence Condition \textup{(S)} holds.
 
These eigenvectors are not normalized, but normalization only rescales the
columns of the squared-eigenvector matrix by nonzero constants.
Computing the determinant gives  $$\det\left(\phi_{k}(j)^2\right)_{1\leq j,k\leq 4} =
\frac{224\sqrt2}{3}\neq 0.$$
 Thus Condition \textup{(I)} also holds.

We now prove that $H_Q$ nevertheless does phase retrieval. Let $\dst u_0=\sum_{k=1}^4a_k\phi_k$.
For every $x\in V$, 
\[
\begin{aligned}
\left|e^{-itH_Q}u_0(x)\right|^2
={}&
\sum_{k=1}^4 |a_k|^2\phi_k(x)^2+2\operatorname{Re}\Bigl[
e^{i(\alpha-\beta)t}
\bigl(
a_1\overline{a_2}\phi_1(x)\phi_2(x)
+a_3\overline{a_4}\phi_3(x)\phi_4(x)
\bigr)\\
&\hspace{7mm}
+e^{i(\alpha+\beta)t}
\bigl(
a_1\overline{a_3}\phi_1(x)\phi_3(x)
+a_2\overline{a_4}\phi_2(x)\phi_4(x)
\bigr)\\
&\hspace{7mm}
+e^{2i\beta t}a_2\overline{a_3}\phi_2(x)\phi_3(x)
+e^{2i\alpha t}a_1\overline{a_4}\phi_1(x)\phi_4(x)
\Bigr].
\end{aligned}
\]

Notice that $0<2\beta<\alpha-\beta<\alpha+\beta<2\alpha$,
so the four positive frequencies appearing above are pairwise distinct. 

Now, let $v_0=\dst\sum_{j=1}^4b_j\phi_j$ and assume 
\[
\left|e^{-itH_Q}u_0(x)\right|^2
=
\left|e^{-itH_Q}v_0(x)\right|^2,
\qquad t\in[0,T],\quad x\in V.
\]
By the linear independence of the corresponding exponential functions, we
obtain the following identities:  
\begin{enumerate}
\renewcommand{\theenumi}{\roman{enumi}}

\item
For every $x\in V$,
\[
\sum_{k=1}^4|a_k|^2\phi_k(x)^2
=
\sum_{k=1}^4|b_k|^2\phi_k(x)^2.
\]
By Condition \textup{(I)}, this is equivalent to
\[
|a_k|=|b_k|,
\qquad k=1,\ldots,4.
\]

\item
For every $x\in V$,
\[
\begin{aligned}
&a_1\overline{a_2}\phi_1(x)\phi_2(x)
+a_3\overline{a_4}\phi_3(x)\phi_4(x) =
b_1\overline{b_2}\phi_1(x)\phi_2(x)
+b_3\overline{b_4}\phi_3(x)\phi_4(x).
\end{aligned}
\]
The matrix
\[
\begin{pmatrix}
\phi_1(1)\phi_2(1)&\phi_3(1)\phi_4(1)\\
\phi_1(2)\phi_2(2)&\phi_3(2)\phi_4(2)
\end{pmatrix}
\]
has determinant
\[
\frac{2}{3}(\alpha+\beta)\neq0.
\]
Therefore
\[
a_1\overline{a_2}=b_1\overline{b_2},
\qquad
a_3\overline{a_4}=b_3\overline{b_4}.
\]

\item
For every $x\in V$,
\[
\begin{aligned}
&a_1\overline{a_3}\phi_1(x)\phi_3(x)
+a_2\overline{a_4}\phi_2(x)\phi_4(x)
 =
b_1\overline{b_3}\phi_1(x)\phi_3(x)
+b_2\overline{b_4}\phi_2(x)\phi_4(x).
\end{aligned}
\]
The matrix
\[
\begin{pmatrix}
\phi_1(1)\phi_3(1)&\phi_2(1)\phi_4(1)\\
\phi_1(2)\phi_3(2)&\phi_2(2)\phi_4(2)
\end{pmatrix}
\]
 is invertible,  therefore 
\[
a_1\overline{a_3}=b_1\overline{b_3},
\qquad
a_2\overline{a_4}=b_2\overline{b_4}.
\]

\item
Comparing the coefficients of $e^{2i\beta t}$ and $e^{2i\alpha t}$, and
using the fact that all the eigenvectors have full support, gives
\[
a_2\overline{a_3}=b_2\overline{b_3},
\qquad
a_1\overline{a_4}=b_1\overline{b_4}.
\]
\end{enumerate}

Together with part \textup{(1)}, these identities give
\[
a_j\overline{a_k}=b_j\overline{b_k},
\qquad 1\leq j,k\leq4.
\]
Thus $aa^*=bb^*$. If $a=0$, then $b=0$. Otherwise, choosing an index $r$
such that $a_r\neq0$ shows that there exists $c\in\mathbb T$ such that
$b_k=ca_k$ for every $k$. Hence
\[
a\sim_{\mathbb T}b
\qquad\text{and therefore}\qquad
u_0\sim_{\mathbb T}v_0.
\]
Thus $H_Q$ does Schr\"odinger phase retrieval.
\end{example}

The next example provides a graph $G$ for which $\Delta_G$ satisfies
conditions $B_2$ and \textup{(S)} and does phase retrieval, although
condition \textup{(I)} fails.

\begin{example}[A graph satisfying $B_2$ and \textup{(S)}, but not
\textup{(I)}, for which $\Delta_G$ does phase retrieval] 
Consider the graph \(G=(V,E)\) defined by
\[
V=\{1,2,3,4,5\}
\quad\mbox{and}\quad
E=
\bigl\{
\{1,2\},\{2,3\},\{1,4\},\{1,5\}
\bigr\}.
\]
\begin{center}
    \begin{tikzpicture}[scale=1.0, every node/.style={circle, draw, minimum size=7mm}]
  \node (1) at (0,0) {$1$};
  \node (2) at (1.5,0) {$2$};
  \node (3) at (3,0) {$3$};
  \node (4) at (-1,1.2) {$4$};
  \node (5) at (-1,-1.2) {$5$};

  \draw (1) -- (2);
  \draw (2) -- (3);
  \draw (1) -- (4);
  \draw (1) -- (5);
\end{tikzpicture}
\end{center}
%
Its five eigenvalues are given by
\[
\lambda_1=0,
\quad
\lambda_2\approx0.518805696,
\quad
\lambda_3=1,
\quad
\lambda_4\approx2.311107817,
\quad
\lambda_5\approx4.170086487
\]
and are easily seen to form a $B_2$-set.
The corresponding eigenvectors are given by
\[
\begin{matrix}
\phi_1&=&\dfrac{1}{\sqrt{5}}&(&1&1&1&1&1&)\\
\phi_2&\approx&&(&
-0.201& 0.337& 0.702&-0.419&
-0.419&)\\
\phi_3&=&\dfrac{1}{\sqrt{2}}&(&0&0&0&1&-1&)\\
\phi_4&=&&(&-0.317&-0.703& 0.536&
 0.242& 0.242&)\\
 \phi_5&=&&(&-0.811& 0.437&
-0.138& 0.255& 0.255&)\\
\end{matrix}.
\]
Note that Condition $(S)$ can be seen in the last two coordinates of those vectors, while
failure of Condition $(I)$
comes from those coordinates.

Now let 
$\dst
u_0=\sum_{k=1}^5 a_k\phi_k\neq0$ and 
$\dst v_0=\sum_{k=1}^5 b_k\phi_k\neq0$ 
be $e^{-it\Delta_G}$-partners.
The proof of Theorem \ref{thm:mainth}
gives us that, 
\[
\sum_{k=1}^5
\bigl(|a_k|^2-|b_k|^2\bigr)\phi_k(j)^2=0,\qquad j\in V,
\]
and
\[
a_k\overline{a_\ell}
=
b_k\overline{b_\ell},
\qquad k\neq\ell.
\]

Set $A=\{k:\ a_k\neq0\}$ and 
$B=\{k\,:\ b_k\not=0\}$, the supports of $a$ and $b$, respectively.
We now distinguish three cases.

\medskip

\noindent
\textbf{Case 1: \(\boldsymbol{|A|=1}\).}

In this case, when $k\not=\ell$
$
b_k\overline{b_\ell}=a_k\overline{a_\ell}=0
$
so that $|B|=1$ as well.
Write $A=\{r\}$, $B=\{s\}$.
Then
\begin{equation}
    \label{eq:newexample}
\sum_{k=1}^5
\bigl(|a_k|^2-|b_k|^2\bigr)\phi_k(j)^2=
|a_r|^2\phi_r(j)^2-|b_s|^2\phi_s(j)^2
=0
\end{equation}
for all $j$.
But then, as the $\phi_k$'s are normalized,
$$
|a_r|^2=|a_r|^2\sum_{j=1}^5\phi_r(j)^2
=|b_s|^2\sum_{j=1}^5\phi_s(j)^2
=|b_s|^2.
$$
Going back to \eqref{eq:newexample},
we obtain $\phi_r(j)^2=\phi_s(j)^2$
for all $j$, which is not the case when $r\not=s$ for this graph.
We conclude that $u_0=a_r\phi_r$,
$v_0=b_r\phi_r$ with $|a_r|=|b_r|$
thus $u_0\trivial v_0$.

Of course, we may exchange the roles
of $A$ and $B$, thus uniqueness also holds when $|B|=1$.
We can thus assume that $|A|,|B|\geq 2$. Note that this implies $A=B$.
Indeed, if   $k,\ell\in A$ and $k\neq\ell$ then,  from $b_k\overline{b_\ell}=a_k\overline{a_\ell}\not=0$, 
we get that $k,\ell\in B$
and conversely.

\medskip

\noindent
\textbf{Case 2: \(\boldsymbol{|A|=2}\).}

Write $A=B=\{r,s\}$, $r\not=s$. We then have
\[
\bigl(|a_r|^2-|b_r|^2\bigr)
\bigl(\phi_r(k)^2\bigr)_{k=1}^5
+
\bigl(|a_s|^2-|b_s|^2\bigr)
\bigl(\phi_s(k)^2\bigr)_{k=1}^5
=0.
\]
But from the above expressions,
the vectors $\bigl(\phi_r(k)^2\bigr)_{k=1}^5$
and $\bigl(\phi_s(k)^2\bigr)_{k=1}^5$
are linearly independent, so that
$|a_r|^2=|b_r|^2$, $|a_s|^2=|b_s|^2$
and, as $a_k=b_k=0$ when $k\notin A$,
we obtain that
\[
a_k\overline{a_\ell}
=
b_k\overline{b_\ell}
\]
holds for all $k,\ell$, even when $k=\ell$. This implies that $a\trivial b$ thus $u_0\trivial v_0$.

\medskip

\noindent
\textbf{Case 3: \(\boldsymbol{|A|\geq3}\).}

The identities
$a_k\overline{a_\ell}
=
b_k\overline{b_\ell}$ imply that $|a_ka_\ell|=|b_kb_\ell|\not=0$ whenever $k,\ell\in A=B$ and $k\neq\ell$.

Fix $k\in A$ and choose distinct indices
$\ell,m\in A\setminus\{k\}$. 
 Then 
$$
|a_k|^2
=
\frac{|a_k\overline{a_\ell}|\,
      |a_k\overline{a_m}|}
     {|a_\ell\overline{a_m}|}
=
\frac{|b_k\overline{b_\ell}|\,
      |b_k\overline{b_m}|}
     {|b_\ell\overline{b_m}|}
=|b_k|^2.
$$
It follows that 
$a_k\overline{a_\ell}
=
b_k\overline{b_\ell},
$ 
also holds when $k=\ell\in A$.

Of course, the identities reduce to \(0=0\)  when $k=\ell\notin A$.
We thus proved that
$$
a_k\overline{a_\ell}
=
b_k\overline{b_\ell},
\qquad k,\ell\in V.
$$
This implies that $a\trivial b$, and 
thus $u_0\trivial v_0$.
\end{example}

The last case can be generalized to more general graphs. Indeed, let
$G=(V,E)$ be a graph with $|V|\geq 3$
such that $\Delta_G$ satisfies Conditions $B_2$ and $(S)$.
Note that such a graph is connected
since $0$ is then a simple eigenvalue.
Let $(\phi_k)_{k=1}^n$ be an orthonormal basis of eigenvectors of $\Delta_G$.
Let $u_0=\dst\sum_{k=1}^na_k\phi_k\not=0$
and $v_0=\dst\sum_{k=1}^nb_k\phi_k\not=0$
be $e^{-it\Delta_G}$-partners.
Let $A=\{k:\ a_k\neq0\}$,
$B=\{k\,:\ b_k\not=0\}$ be the supports of $a$ and $b$.
 Then $a,b$ satisfy the equations 
\begin{equation}
    \label{eq:mixedtermsblabla}
a_k\overline{a_\ell}
=
b_k\overline{b_\ell},
\qquad k\not=\ell\in V.
\end{equation}

But then Case 3 in the above proof shows that, if $|A|\geq 3$, $A=B$
and further \eqref{eq:mixedtermsblabla} also
holds when $k=\ell$ which implies
$a\trivial b$.
This shows the following:

\begin{theorem}\label{thm:almostall}
    Let $G$ be a finite graph such
    that $\Delta_G$ satisfies Conditions $B_2$ and $(S)$. Assume that $|V|\geq 3$.
    Then  for Lebesgue-almost every $u_0\in\C^V$,
    if $v_0$ is an  
 $e^{-it\Delta_G}$-partner of $u_0$,  then $v_0\trivial u_0$.
\end{theorem}

\subsection{Finite time sampling}

In the proof of Theorem~\ref{thm:mainth}, the continuous time interval
$[0,T]$ was used only to separate finitely many exponential functions.
This observation allows us to replace the full interval of observation times
by a finite set of time samples.

Let $\mu_1,\ldots,\mu_N$ be distinct real numbers. Choose $h>0$ such that
$(N-1)h\leq T$
and such that the numbers $e^{i\mu_rh}$ are pairwise distinct. This is possible,
since the forbidden values of $h$ form a discrete set. Then the Vandermonde
matrix $\bigl[e^{i\mu_rh(m-1)}\bigr]_{r,m=1}^N$ is invertible. Therefore
\[
\sum_{r=1}^N c_r e^{i\mu_rt}=0
\quad\text{for }t=0,h,\ldots,(N-1)h
\]
implies $c_1=\cdots=c_N=0$. Thus, replacing Lemma \ref{lem:linind} by this fact in the proof of Theorem
\ref{thm:mainth}, we obtain the following:

\begin{proposition}[Finite time sampling]\label{prop:finite-time-sampling}
Assume that $H_Q$ satisfies Condition $B_2$. Then there exists a finite set
of times $\{t_1,\ldots,t_N\}\subset[0,T]$ such that the full data
$$
\bigl(|e^{-itH_Q}u(j)|\bigr)_{t\in[0,T],\,j\in V}
$$
is determined by the sampled data 
$$
\bigl(|e^{-it_mH_Q}u(j)|\bigr)_{m=1,\ldots,N,\,j\in V}.
$$
\end{proposition}

\subsection{A frame-measurement variant}

The argument above also applies when the vertex measurements are replaced by
measurements against a fixed frame and when the Schr\"odinger operator
is replaced by a more general symmetric operator $H$ on an $n$-dimensional Hilbert space
$\mathcal H$. We include this variant for completeness.

Let $\Psi=(\psi_m)_{m=1}^M$ be a frame for $\mathcal H$.
We denote by $(\phi_k)_{k=1}^n$ the orthonormal eigenbasis of $H$ and the corresponding eigenvalues by
$(\lambda_k)_{k=1}^n$. 
Set
\[
\gamma_{k,m}=\langle \phi_k,\psi_m\rangle .
\]
We say that $H$ satisfies Condition $\textup{(I)}_\Psi$ relative to $\Psi$ if the
matrix $A_\Psi=\big(|\gamma_{k,m}|^2\big)_{k,m}$
has rank $n$. We say that $H$ satisfies Condition $\textup{(S)}_\Psi$ relative
to $\Psi$ if, for every $k\neq \ell$, there exists $m$ such that
$\gamma_{k,m}\overline{\gamma_{\ell,m}}\neq 0$.

With these notations, we have

\begin{theorem}[Dynamical phase retrieval from frame measurements]
\label{thm:frame-measurement}
Let $\Psi=(\psi_m)_{m=1}^M$ be a frame.
Let $H$ be as above, and assume that its eigenvalues satisfy Condition $B_2$
and that its eigenvectors satisfy Conditions $\textup{(I)}_\Psi$
and $\textup{(S)}_\Psi$ relative to $\Psi$. If $u_0,v_0\in\mathcal H$ satisfy
\begin{equation}
    \label{eq:probframe}
|\langle e^{-itH}u_0,\psi_m\rangle|
=
|\langle e^{-itH}v_0,\psi_m\rangle|,
\qquad
t\in[0,T],\quad m=1,\dots,M,
\end{equation}
then $u_0\sim_{\mathbb T} v_0$.
\end{theorem}

Theorem \ref{thm:mainth} is the special case $\psi_m=\delta_m$. Then
\[
\gamma_{k,m}=\langle \phi_k, \psi_m \rangle=\phi_k(m),
\]
so Condition $\textup{(I)}_\Psi$ is precisely the invertibility of the squared
eigenvector matrix, {\em i.e.,} Condition \textup{(I)}, and Condition $\textup{(S)}_\Psi$ is the support condition
for pairs of eigenvectors  \textup{(S)}. 
The proof of Theorem~\ref{thm:mainth} applies verbatim.

Condition $\textup{(I)}_\Psi$  is used to guarantee that, if $u_0=\dst\sum_{k=1}^na_k\phi_k$
and $v_0=\dst\sum_{k=1}^nb_k\phi_k$ are such that \eqref{eq:probframe} holds, then
$|a_k|=|b_k|$ for all $k$.
This is automatically satisfied when the frame $\Psi$ contains
    the eigenbasis $(\phi_k)_{k=1}^n$ since
    $$
    |a_k|=|\langle e^{-itH}u_0,\phi_k\rangle|=|\langle e^{-itH}v_0,\phi_k\rangle|=|b_k|.    
$$

\subsection{A reformulation of Condition  \textup{(I)} when the eigenvalues are simple}\label{sec:reform}

Assume that $H_Q$ has simple eigenvalues. 
Let
$(\phi_k)_{k=1}^n$ be a real orthonormal basis of eigenvectors of $H_Q$,
with corresponding eigenvalues $(\lambda_k)_{k=1}^n$. Define
\[
B_Q=\bigl[(H_Q^{m-1})_{j,j}\bigr]_{1\leq j,m\leq n},
\]
that is, the $(j,m)$-entry of $B_Q$ is the $j$-th diagonal entry of
$H_Q^{m-1}$.  

By the spectral decomposition,  
$$ 
H_Q=\sum_{k=1}^n\lambda_k\phi_k\phi_k^T\quad\mbox{thus}\quad
H_Q^\ell=\sum_{k=1}^n\lambda_k^\ell\phi_k\phi_k^T
\quad\mbox{for any }\ell\geq 0.
$$

Taking the $j$-th diagonal entry gives
\[
(H_Q^\ell)_{j,j}
=
\sum_{k=1}^n \lambda_k^\ell \phi_k(j)^2.
\]  
Hence
\[
B_Q=MV,
\quad\mbox{where}\quad
M=\bigl[\phi_k(j)^2\bigr]_{1\leq j,k\leq n},
\qquad
V=\bigl[\lambda_k^{m-1}\bigr]_{1\leq k,m\leq n}.
\]

Since the eigenvalues are simple, the Vandermonde matrix $V$ is invertible.
 Therefore $M$ is invertible if and only if $B_Q$ is invertible. 
In other
words, under the simple-spectrum assumption, Condition \textup{(I)} is
equivalent to the invertibility of the matrix
\[
B_Q=\bigl[(H_Q^{m-1})_{j,j}\bigr]_{1\leq j,m\leq n}.
\]
This gives a reformulation of Condition \textup{(I)} that does not require
computing the eigenvectors.

\section{Explicit and generic perturbations with diagonal potentials for pha\-se retrieval }\label{sec:diagonal-potentials}

The spectral criterion becomes useful once one knows that the three spectral
conditions can actually occur. In this section,  we prove this in two ways. First
we give an explicit {\it universal} diagonal potential on every connected graph. We then
show that the same conditions hold generically under diagonal perturbations of the Laplacian.

\subsection{An explicit good potential on connected graphs}

Let $G=(V,E)$ be a connected graph with $|V|=n$. Fix a vertex $a\in V$ and
label the vertices 
\[
V=\{v_1,\ldots,v_n\},\qquad v_1=a,
\]
so that
\[
\dist(a,v_i)\leq \dist(a,v_j)
\qquad\text{whenever } i<j.
\]
Let $A$ be the adjacency matrix and let $D_G$ be the diagonal degree matrix, so
that
\[
\Delta_G=D_G-A.
\]
For $\rho>1$, set
\[
\Lambda_\rho=\diag(\rho,\rho^2,\ldots,\rho^n),
\qquad
Q_\rho=\Lambda_\rho-D_G.
\]
Then
\[
H_{Q_\rho}
=
\Delta_G+Q_\rho
=
\Lambda_\rho-A.
\]

\begin{proposition}\label{prop:explicit-good-potential-connected}
There exists $R=R(G)>0$ such that, for every $\rho>R$, the operator
\[
H_{Q_\rho}=\Lambda_\rho-A
\]
satisfies Conditions $B_2$--\textup{(I)}--\textup{(S)}.
\end{proposition}

\begin{proof}
Write
\[
H_\rho:=\Lambda_\rho-A,
\qquad
M:=\|A\|, 
\]
where $\| \cdot \|$ is the spectral norm.    
Then $M\leq n-1$. We may label the eigenvalues $(\lambda_j)_{j=1}^n$ of $H_\rho$
increasingly, and then
\[
|\lambda_j-\rho^j|\leq M,
\qquad j=1,\ldots,n
\]
by Weyl's inequality (see \cite[Theorem 4.3.1]{HornJohnson}). For $\rho$ large, the intervals
\[
[\rho^j-M,\rho^j+M],\qquad j=1,\ldots,n,
\]
are pairwise disjoint, and hence the eigenvalues are simple and ordered
according to the diagonal entries.

\smallskip

\noindent\emph{Condition $B_2$.}
Suppose
\[
\lambda_j+\lambda_k=\lambda_{j'}+\lambda_{k'}.
\]
Then
\[
\bigl|(\rho^j+\rho^k)-(\rho^{j'}+\rho^{k'})\bigr|
\leq 4M.
\]
On the other hand, two distinct numbers of the form
\[
\rho^a+\rho^b,\qquad 1\leq a\leq b\leq n,
\]
differ by at least $\rho(\rho-2)$ for $\rho\geq 3$. Indeed, assume
\[
\rho^a+\rho^b<\rho^c+\rho^d,
\qquad a\leq b,\quad c\leq d.
\]
If $b<d$, then
\[
\rho^c+\rho^d-(\rho^a+\rho^b)
\geq
\rho^d-\rho^b-\rho^a
\geq
\rho^{b+1}-2\rho^b
=
\rho^b(\rho-2)
\geq
\rho(\rho-2).
\]
If $b=d$, then $a<c$, and
\[
\rho^c+\rho^d-(\rho^a+\rho^b)
=
\rho^c-\rho^a
\geq
\rho^a(\rho-1)
\geq
\rho(\rho-1).
\]
Choosing $\rho$ so large that $\rho(\rho-2)>4M$, we get $\{j,k\}=\{j',k'\}$.
Thus the spectrum of $H_\rho$ satisfies Condition $B_2$.

\smallskip

\noindent\emph{Condition  \textup{(I)}.}
Let $\phi_j$ be a real normalized eigenvector of $H_\rho$ associated with $\lambda_j$.
Write
\[
\phi_j=\alpha_j e_j+\eta_j,
\qquad
\eta_j=P_j^\perp\phi_j,
\]
where $P_j$ is the orthogonal projection onto $\C e_j$ and
$P_j^\perp=I-P_j$. Set
\[
\delta_j(\rho)=\min_{k\neq j}|\rho^k-\rho^j|.
\]
Then
\[
\min_j\delta_j(\rho)\to+\infty
\qquad\text{as }\rho\to+\infty.
\]

Projecting
\(
(H_\rho-\lambda_j)\phi_j=0
\)
onto $e_j^\perp$ gives
\[
\bigl(P_j^\perp H_\rho P_j^\perp-\lambda_j\bigr)\eta_j
=
\alpha_j P_j^\perp A e_j.
\]
Moreover,
\[
P_j^\perp H_\rho P_j^\perp
=
P_j^\perp\Lambda_\rho P_j^\perp
-
P_j^\perp A P_j^\perp .
\]
Since
\[
\dist\bigl(\lambda_j,\sigma(P_j^\perp\Lambda_\rho P_j^\perp)\bigr)
\geq
\delta_j(\rho)-M
\]
and
\[
\|P_j^\perp A P_j^\perp\|\leq M,
\]
we get
\[
\dist\bigl(\lambda_j,\sigma(P_j^\perp H_\rho P_j^\perp)\bigr)
\geq
\delta_j(\rho)-2M.
\]
Thus, for $\rho$ large,
\(
P_j^\perp H_\rho P_j^\perp-\lambda_j
\)
is invertible on $e_j^\perp$, and
\[
\left\|
\bigl(P_j^\perp H_\rho P_j^\perp-\lambda_j\bigr)^{-1}
\right\|
\leq
\frac{1}{\delta_j(\rho)-2M}.
\]
Therefore
\[
\|\eta_j\|
\leq
\frac{\|Ae_j\|}{\delta_j(\rho)-2M}\,|\alpha_j|
\leq
\frac{\sqrt{n-1}}{\delta_j(\rho)-2M}\,|\alpha_j|.
\]
Put
\[
c_\rho
=
\max_{1\leq j\leq n}
\frac{\sqrt{n-1}}{\delta_j(\rho)-2M}.
\]
Then $c_\rho\to0$, and
\[
\|\eta_j\|\leq c_\rho|\alpha_j|.
\]
Since $\|\phi_j\|=1$,
\[
1=|\alpha_j|^2+\|\eta_j\|^2
\leq
(1+c_\rho^2)|\alpha_j|^2.
\]
Hence
\[
|\alpha_j|^2\geq \frac{1}{1+c_\rho^2},
\qquad
\|\eta_j\|^2\leq \frac{c_\rho^2}{1+c_\rho^2}.
\]

Now consider
\[
B=\bigl(\phi_j(k)^2\bigr)_{1\leq j,k\leq n}.
\]
The transpose of $B$ is the squared-eigenvector matrix from Condition  \textup{(I)}.
For each $j$,
\[
B_{jj}=\phi_j(j)^2=|\alpha_j|^2
\geq
\frac{1}{1+c_\rho^2},
\]
while
\[
\sum_{k\neq j}|B_{jk}|
=
\sum_{k\neq j}\phi_j(k)^2
=
\|\eta_j\|^2
\leq
\frac{c_\rho^2}{1+c_\rho^2}.
\]
For $\rho$ large, $c_\rho<1$, and therefore
\[
B_{jj}>\sum_{k\neq j}|B_{jk}|.
\]
Thus $B$ is strictly diagonally dominant. By the Levy--Desplanques theorem 
\cite[Theorem~6.2.27]{HornJohnson}
$B$ is invertible. Hence, Condition  \textup{(I)} holds.

\smallskip

\noindent\emph{Condition  \textup{(S)}.}
It remains to show that all eigenvectors share a common nonzero coordinate. The
estimate above gives $\phi_j(j)\neq 0$ for $\rho$ large. We now rescale the
eigenvector associated with $\lambda_j=\lambda_j(\rho)$ by
\[
\widetilde{\phi}_j(v_j)=1.
\]
Write
\[
\widetilde{\phi}_j=e_j+\eta_j,
\qquad
\eta_j\in e_j^\perp.
\]
As before,
\[
\bigl(P_j^\perp H_\rho P_j^\perp-\lambda_j\bigr)\eta_j
=
P_j^\perp A e_j.
\]
Set
\[
T_j(\rho)
=
P_j^\perp\Lambda_\rho P_j^\perp-\lambda_j P_j^\perp,
\qquad
B_j=P_j^\perp A P_j^\perp .
\]
Then
\[
P_j^\perp H_\rho P_j^\perp-\lambda_j
=
T_j(\rho)-B_j.
\]
For $\rho$ large,
\[
\|T_j(\rho)^{-1}B_j\|<1,
\]
and hence
\[
(T_j(\rho)-B_j)^{-1}
=
\sum_{\ell=0}^{\infty}
\bigl(T_j(\rho)^{-1}B_j\bigr)^\ell T_j(\rho)^{-1}.
\]
Thus
\[
\eta_j
=
\sum_{\ell=0}^{\infty}
\bigl(T_j(\rho)^{-1}B_j\bigr)^\ell
T_j(\rho)^{-1}P_j^\perp A e_j .
\]

We take the $v_1$-coordinate. If $j=1$, then
\[
\widetilde{\phi}_1(v_1)=1.
\]
Assume $j\geq2$, and set
\[
r=\dist(v_1,v_j).
\]
Taking the $v_1$-coordinate in the Neumann expansion gives the usual walk
expansion. The first nonzero terms come from shortest paths
\[
v_1=w_0\sim w_1\sim\cdots\sim w_{r-1}\sim v_j.
\]
A shortest path of this form contributes
\[
\prod_{m=0}^{r-1}
\frac{1}{\rho^{i(w_m)}-\lambda_j},
\]
where $i(w_m)$ denotes the label of the vertex $w_m$.

Along a shortest path from $v_1$ to $v_j$,
\[
\dist(v_1,w_m)=m<r=\dist(v_1,v_j),
\]
and therefore, by the choice of the labeling,
\[
i(w_m)<j,
\qquad m=0,\ldots,r-1.
\]
Hence
\[
\rho^{i(w_m)}-\lambda_j
=
-\rho^j(1+o(1)).
\]
Thus every shortest path contributes
\[
(-1)^r\rho^{-jr}(1+o(1)).
\]
All shortest-path contributions have the same sign. Since $G$ is connected,
there is at least one shortest path from $v_1$ to $v_j$. If $N_j$ is the number
of such paths, then
\[
\widetilde{\phi}_j(v_1)
=
(-1)^r N_j\rho^{-jr}
+
o(\rho^{-jr}).
\]
Since $N_j\geq1$, the leading coefficient is nonzero. Hence
\[
\widetilde{\phi}_j(v_1)\neq0
\]
for all sufficiently large $\rho$.

There are only finitely many $j$. After increasing $R$ if necessary, we have
\[
\widetilde{\phi}_j(v_1)\neq0
\qquad\text{for every }j=1,\ldots,n.
\]
Thus all eigenvectors have a common nonzero coordinate. Therefore,  the supports of every pair of eigenvectors intersect, and Condition  \textup{(S)} holds.
\end{proof}

\subsection{Generic potentials}

Theorem~\ref{thm:mainth} gives sufficient conditions for uniqueness in the
dynamical phase retrieval problem. It is therefore natural to ask whether
these conditions are stable under perturbations of the potential. In this
section, we prove Theorem~\ref{thm:generic}, which shows that, on every
connected graph, the conditions hold for almost every diagonal potential.

Let $G=(V,E)$ be a connected graph. 
It is enough to show 
that, for almost every $Q\in\R^V$, $H_Q$
satisfies each  of the three properties
$B_2$, \textup{(I)} and \textup{(S)}. This will be done respectively in Propositions \ref{prop:generic-B2}, \ref{prop:generic-I}
and \ref{thm:generic-S}.
It follows that
$H_Q$ satisfies the three Conditions $B_2$--\textup{(I)}--\textup{(S)} simultaneously for almost every $Q\in\R^V$.
From Theorem \ref{thm:mainth}, we then conclude that $H_Q$ does Schr\"odinger phase retrieval, concluding the proof of Theorem \ref{thm:generic}.

The proofs of the three propositions give slightly stronger results as the set
of $Q$'s to be excluded in each of these propositions is a proper algebraic subset of $\R^V$.

\subsubsection{Genericity of Condition $B_2$}\label{sec:b2}
First,  we show that  Condition $B_2$ holds
generically for Schr\"o\-din\-ger operators with random diagonal potentials. We will prove a stronger result that holds for any real symmetric matrix $L$.

For $Q\in\mathbb R^n$, we consider
\begin{equation}
    \label{def:LQ}
L_Q=L+\diag(Q_1,\ldots,Q_n).
\end{equation}
When $Q$ is chosen according to a probability measure absolutely continuous
with respect to Lebesgue measure,  we prove that, almost surely, the spectrum of
$L_Q$ is a $B_2$-set. This applies in particular when $L=\Delta_G$ so that
$L_Q=H_Q$.

We begin with a standard symmetric-tensor reformulation.
\begin{definition}\label{defn-sym-tensor}
Let $\R^n\otimes\R^n$ be the tensor product, and let
\[
\tau:\R^n\otimes\R^n\to\R^n\otimes\R^n\quad,\quad
\tau(x\otimes y)=y\otimes x
\]
be the flip map. We define
\[
\operatorname{Sym}^2(\R^n)
=
\{w\in\R^n\otimes\R^n:\tau w=w\}.
\]
Equivalently, if $e_1,\ldots,e_n$ is the standard basis of $\R^n$, then
$\operatorname{Sym}^2(\R^n)$ is spanned by
\[
e_i\otimes e_i,
\qquad
1\leq i\leq n,
\quad\mbox{and}
\quad
e_i\otimes e_j+e_j\otimes e_i,
\qquad
1\leq i<j\leq n.
\]
\end{definition}

If $L$ is a real symmetric matrix, then
\[
L\otimes I+I\otimes L
\]
preserves $\operatorname{Sym}^2(\R^n)$. If
\[
L\phi_k=\lambda_k\phi_k,
\qquad k=1,\ldots,n,
\]
then the eigenvalues of the restriction of $L\otimes I+I\otimes L$ to
$\operatorname{Sym}^2(\R^n)$ are
\[
\lambda_i+\lambda_j,
\qquad
1\leq i\leq j\leq n.
\]
Therefore, the eigenvalues of $L$ form a $B_2$-set if and only if this
restricted operator has simple spectrum.
This construction will also be used in Section \ref{sec:obs3}.

\begin{proposition}[Genericity of Condition \(B_2\)]\label{prop:generic-B2}
Let $L$ be a fixed real symmetric $n\times n$ matrix, $Q\in\R^n$ and define $L_Q$
as in \eqref{def:LQ}.
Assume that $Q$ is distributed according to a probability measure
absolutely continuous with respect to Lebesgue measure on $\R^n$.
Then, almost surely, the spectrum of $L_Q$ is a $B_2$-set.
\end{proposition}

\begin{proof}
Let $\lambda_1(Q),\ldots,\lambda_n(Q)$ be the eigenvalues of $L_Q$, counted
with multiplicity, and put
\[
        K_Q=L_Q\otimes I+I\otimes L_Q
        \quad\text{on } \operatorname{Sym}^2(\mathbb R^n).
\]
Then $L_Q$ has $B_2$ spectrum if and only if $K_Q$ has
simple spectrum.

Fix a basis of $\operatorname{Sym}^2(\mathbb R^n)$ and set
\[
        P_Q(X)=\det(XI-K_Q).
\]
The coefficients of $P_Q$ are polynomial in $Q$, and the exceptional set is
$\{Q\,:\operatorname{Disc}(P_Q)=0\}$.
It remains to check that this discriminant is not identically zero.

Choose a finite $B_2$-set $\alpha=\{\alpha_1,\ldots,\alpha_n\}$, for instance
$\alpha_r=3^r$. Then
\[
        t^{-1}L_{t\alpha}
        =
        \operatorname{diag}(\alpha_1,\ldots,\alpha_n)+t^{-1}L,
\]
so, after relabeling,
\[
        \lambda_r(t\alpha)=t\alpha_r+O(1).
\]
Hence
\[
        \lambda_i(t\alpha)+\lambda_j(t\alpha)
        =
        t(\alpha_i+\alpha_j)+O(1).
\]
Since the sums $\alpha_i+\alpha_j$, $1\le i\le j\le n$, are pairwise
distinct, the same holds for
$\lambda_i(t\alpha)+\lambda_j(t\alpha)$ when $t$ is large. Thus $K_{t\alpha}$
has simple spectrum for some $t$, and therefore
\[
        \operatorname{Disc}(P_Q)\not\equiv 0.
\]
Its zero set is a proper algebraic subset of $\mathbb R^n$, hence has
Lebesgue measure zero. Absolute continuity of the law of $Q$ then gives
\[
        \mathbb P\{\operatorname{Disc}(P_Q)=0\}=0.
\]
Equivalently, almost surely, $K_Q$ has simple spectrum, i.e., the spectrum of
$L_Q$ is a $B_2$-set.
\end{proof}

\subsubsection{Genericity of Condition \textup{(I)}}\label{sec:genI}

Let us first note that  if G is connected, then 
Condition \textup{(I)} is never satisfied by the unperturbed Laplacian $\Delta_G$, so  here the
perturbation with a potential is needed. 
 Indeed, since the zero eigenvalue is simple, every orthonormal eigenbasis contains the normalized constant eigenvector 
$\phi_1=n^{-1/2}(1,1,\ldots,1)$.

Consider then $\phi_2,\ldots,\phi_n$
so that $(\phi_k)_{k=1,\ldots,n}$ is an orthonormal basis of (real) eigenvectors of $\Delta_G$.
Writing the standard basis $(e_j)_{j=1,\ldots,n}$ in this basis and using Parseval, we get
$$
1=\|e_j\|^2=\sum_{k=1}^n\scal{e_j,\phi_k}^2=\sum_{k=1}^n\phi_k(j)^2.
$$
Isolating $\phi_1(j)^2=n^{-1}$ gives
$$
\sum_{k=2}^n\phi_k(j)^2=1-\frac{1}{n}=\frac{n-1}{n}=(n-1)\phi_1(j)^2.
$$
This shows that the columns of the matrix $[\phi_k(j)^2]$ are linearly dependent so that this matrix is not invertible.
In other words,  $\Delta_G$ does {\em not} satisfy Condition \textup{(I)}. 

\smallskip

Let us now prove that Condition \textup{(I)} holds under generic perturbations of $\Delta_G$.

 \begin{proposition}
\label{prop:generic-I}
Let $G=(V,E)$ be a finite graph with $|V|=n$. Then, for Lebesgue-a.e.
$Q\in\R^V$, the operator $H_Q$ satisfies Condition \textup{(I)}.
\end{proposition}

\begin{proof}
By Proposition~\ref{prop:generic-B2}, for almost every $Q\in\R^V$, the
eigenvalues $(\lambda_k)_{k=1}^n$ of $H_Q$ form a $B_2$-set and are
therefore simple. For such a $Q$, let $(\phi_k)_{k=1}^n$ be a corresponding
real orthonormal basis of eigenvectors, and set
\[
M_Q=\bigl[\phi_k(j)^2\bigr]_{j,k=1}^n.
\]
We want to prove that $M_Q$ is invertible.

We use the reformulation from Section~\ref{sec:reform}. Consider the
Vandermonde matrix $V_Q=\bigl[\lambda_k^{m-1}\bigr]_{k,m=1}^n$ and the matrix
$B_Q=\bigl[(H_Q^{m-1})_{j,j}\bigr]_{j,m=1}^n$. We have shown that $B_Q=M_QV_Q$.

Since the eigenvalues are simple, $V_Q$ is invertible. Hence, for such $Q$,
$M_Q$ is invertible if and only if $B_Q$ is invertible.

Define
\[
F(Q)=\det B_Q.
\]
Since every entry of $B_Q$ is a polynomial in the coordinates of $Q$, the
function $F$ is a polynomial on $\R^V$. We show that $F\not\equiv 0$.

Fix a labeling $V=\{1,\ldots,n\}$. Choose distinct real numbers
$c_1,\ldots,c_n$, and set
\[
c=(c_1,\ldots,c_n)\in\R^V,
\qquad
C=\operatorname{diag}(c_1,\ldots,c_n).
\]
We first consider the particular case $Q=tc$ when $t\to+\infty$, thus
\[
H_{tc}=\Delta_G+tC.
\]
For every fixed integer $r\geq 1$, 
\[
H_{tc}^r=t^rC^r+O(t^{r-1}).
\]
Consequently, for $m\geq 2$,
\[
(H_{tc}^{m-1})_{j,j}
=
t^{m-1}c_j^{m-1}+O(t^{m-2}),
\]
while the first column of $B_{tc}$ is exactly $(1,\ldots,1)^T$. Therefore
\[
F(tc)
=
t^{0+1+\cdots+(n-1)}
\left(
\det\bigl[c_j^{m-1}\bigr]_{j,m=1}^n+O(t^{-1})
\right).
\]
Since
\[
\det\bigl[c_j^{m-1}\bigr]_{j,m=1}^n
=
\prod_{1\leq i<j\leq n}(c_j-c_i)\neq 0,
\]
we have $F(tc)\neq 0$ for all sufficiently large $t$. Hence
$F\not\equiv 0$.

It follows that
\[
Z=\{Q\in\R^V:F(Q)=0\}
=
\{Q\in\R^V:B_Q\text{ is singular}\}
\]
is a proper algebraic subset of $\R^V$, and therefore has Lebesgue measure
zero. The fact used here is that the zero set of a nonzero real polynomial
has Lebesgue measure zero.  

Outside the union of $Z$ and the measure zero set where the eigenvalues of
$H_Q$ fail to form a $B_2$-set, both $B_Q$ and $V_Q$ are invertible.
Since
\[
B_Q=M_QV_Q,
\]
the matrix $M_Q$ is invertible. Hence $H_Q$ satisfies Condition
\textup{(I)} for Lebesgue-a.e. $Q\in\R^V$.
\end{proof}

\begin{remark}
Connectedness was not used in the proofs of Proposition \ref{prop:generic-B2} and \ref{prop:generic-I}.  Thus Condition \textup{(I)} is generic for
arbitrary finite graphs under diagonal perturbations; connectedness enters only
in the generic full-support Condition \textup{(S)}.
\end{remark}

\subsubsection{Genericity of Condition \textup{(S)}}

\begin{proposition}[Genericity of Condition $\textup{(S)}$]
\label{thm:generic-S}
Let \(G=(V,E)\) be a finite connected graph.
Then, for Lebesgue-a.e. \(Q\in\R^V\), every eigenvector of \(H_Q\) has full
support. In particular, \(H_Q\) satisfies Condition \textup{(S)} for a.e. \(Q\).
\end{proposition}

The proof uses a controllability argument, inspired by
\cite{Godsil2012ControllableSubsets,MonfaredShader2016NowhereZero,
LeveneOblakSmigoc2024GenericEigenvectors,PoignardPereiraPade2018LaplacianSpectra}.

\begin{proof}
Fix \(a\in V\), and set
\[
        \mathcal K_a(Q)
        =
        \big[e_a,H_Qe_a,\ldots,H_Q^{n-1}e_a\big],
        \qquad
        P_a(Q)=\det \mathcal K_a(Q).
\]
Then \(P_a\) is a polynomial in \(Q\). We first show \(P_a\not\equiv 0\).

Choose distinct numbers \(c_v\),  one for each  \(v\in V\), such that 
\[
        \dist(a,u)<\dist(a,v)\quad\Longrightarrow\quad c_u<c_v.
\]
For \(t\) large, the eigenvalues of $H_{tc}$ are simple, and the eigenvectors are
perturbations of the standard basis. Let \(u_b(t)\) be the eigenvector
corresponding to the eigenvalue asymptotic to \(tc_b\), normalized by
\(u_b(t,b)=1\). Writing \(r=\dist(a,b)\), the usual path expansion gives
\[
        u_b(t,a)
        =
        t^{-r}
        \sum_{\gamma:b\to a}
        \prod_{m=1}^{r}\frac{1}{c_{v_m}-c_b}
        +
        O(t^{-r-1}),
\]
where the sum is over shortest paths
\(\gamma=(b=v_0,v_1,\ldots,v_r=a)\). For \(b=a\), the sum is interpreted as
\(1\).

Along every shortest path from \(b\) to \(a\), the distance to \(a\) strictly
decreases, hence
\[
        c_{v_m}<c_b,\qquad m=1,\ldots,r.
\]
Thus all summands in the leading coefficient have the same sign, and the
coefficient is nonzero. Hence \(u_b(t,a)\neq 0\) for all large \(t\), for every
\(b\in V\).

It follows that, for such \(t\), \(e_a\) has nonzero projection onto every
eigenspace of \(H_{tC}\). Since the spectrum is simple, \(e_a\) is cyclic, i.e.
\[
        \det\big[e_a,H_{tC} e_a,\ldots,H_{tC}^{n-1}e_a\big]\neq 0.
\]
This gives \(P_a(tc)\neq0\) for some \(t\), and therefore
\(P_a\not\equiv0\). Hence
\[
        Z_a:=\{Q\in\R^V:P_a(Q)=0\}
\]
is a proper algebraic set, thus has measure zero.

Let
\[
        Z=\bigcup_{a\in V} Z_a .
\]
Then \(Z\) has measure zero. If \(Q\notin Z\), then for every \(a\in V\),
\[
        e_a,H_Qe_a,\ldots,H_Q^{n-1}e_a
\]
is a basis of \(\R^V\).

Now let \(H_Qu=\lambda u\). If \(u(a)=0\), then for \(k=0,\ldots,n-1\),
\[
        \langle u,H_Q^k e_a\rangle
        =
        \langle H_Q^k u,e_a\rangle
        =
        \lambda^k \langle u,e_a\rangle
        =0.
\]
Thus \(u\) is orthogonal to a basis, hence \(u=0\), a contradiction. Therefore
\(u(a)\neq0\). Since \(a\) was arbitrary, every eigenvector has full support for
all \(Q\notin Z\).
\end{proof}

\section{Obstructions}\label{sec:obstructions}

\subsection{Disconnected graphs, multiple eigenvalues}

Let $G=(V,E)$ be a simple graph.
Suppose that $G$ has two connected components $V_1=\{1,\ldots,m\}$ and $V_2=\{m+1,\ldots,n\}$ and denote by $G_1,G_2$ the corresponding induced graphs. Let $u_1\in\C^{V_1},u_2\in\C^{V_2}$, $u_1,u_2\not=0$, and $u_0=(u_1,u_2)\in\C^V$. 
Then the restriction of $e^{-it\Delta_G}u_0$ to $V_j$ is $e^{-it\Delta_{G_j}}u_j$.
In particular, if $c\in\T$, $c\not=1$, and $v_0=(u_1,cu_2)$ then
$u_0,v_0$ are $e^{-it\Delta_G}$-partners but nontrivial partners.
The construction extends to an arbitrary number of connected components
and also applies when $Q\not=0$.

\smallskip

Next assume that $H_Q$ has an eigenvalue $\lambda$ of multiplicity $\geq 2$ and let $E_\lambda$
be the corresponding eigenspace. If $u_0\in E_\lambda$ then $e^{-itH_Q}u_0=e^{-it\lambda}u_0$.
In particular, for every $j\in V$, 
$$
|e^{-itH_Q}u_0(j)|=|e^{-it\lambda}u_0(j)|=|u_0(j)|=|\overline{u_0(j)}|=
|e^{-itH_Q}\overline{u_0(j)}|
$$
so that $\overline{u_0}$ is an $\partner$-partner of $u_0$. However, every complex space of dimension at 
least $2$ contains vectors $u_0$ such that $\overline{u_0}$ is not a multiple of $u_0$. 

\smallskip

It is easy to build graphs and potentials for which $H_Q$ has repeated eigenvalues. To do so,
recall that the Cartesian product \(G\square G'\) has vertex set
\(V\times V'\), and
\[
        (j,j')\sim(k,k')
\]
if either \(j\sim k\) in \(G\) and \(j'=k'\), or \(j=k\) and
\(j'\sim k'\) in \(G'\). Its Laplacian is
\[
        \Delta_{G\square G'}
        =
        \Delta_G\otimes I_{G'}
        +
        I_G\otimes \Delta_{G'} .
\]
Here we use the notation introduced in Section \ref{sec:b2}.
Next, if \(Q:V\to\mathbb R\) and \(Q':V'\to\mathbb R\), and $R\,:V\times V'\to\R$ is defined by
$R(j,j')=Q(j)+Q'(j')$,
then $H_R=\Delta_{G\square G'}+R$ is given by
\[
        H_R
        =
        H_Q\otimes I_{G'}
        +
        I_G\otimes H_{Q'} .
\]
Therefore,
\[
        e^{-itH_R}(u\otimes u')
        =
        \bigl(e^{-itH_Q}u\bigr)\otimes
        \bigl(e^{-itH_{Q'}}u'\bigr).
\]
As a direct consequence, if \(H_R\) does Schr\"odinger phase retrieval 
then both \(H_Q\) and \(H_{Q'}\) do Schr\"odinger phase retrieval.  
%
The converse fails. Indeed, take $G=G'$ and $Q=Q'$. Then if $(\phi_j)_{j=1}^n$ are the eigenvectors
of $H_Q$ with corresponding eigenvalues $(\lambda_j)_{j=1}^n$ then
$H_R\phi_j\otimes\phi_k=(\lambda_j+\lambda_k)\phi_j\otimes\phi_k$.
Therefore the eigenvalues of $H_R$ are given by
$\{\lambda_j+\lambda_k:1\le j,k\le n\}$. 
In particular, if \(j\neq k\), then
\[
        \lambda_j+\lambda_k=\lambda_k+\lambda_j,
\]
and the eigenvalues now have multiplicity. It follows that $H_R$ fails to do phase retrieval.

\subsection{Failure of Condition  \textup{(S)}}

\begin{lemma}
    Let $G=(V,E)$ be a graph, and let $Q\in\mathbb R^V$ be such that $H_Q$  
    does not satisfy Condition  \textup{(S)}. 
Then $H_Q$ fails to do phase retrieval.
\end{lemma}

\begin{proof}
Let $(\phi_k)_{k=1}^n$, $(\lambda_k)_{k=1}^n$ be the eigenvectors and eigenvalues of $H_Q$.
As $H_Q$ does not satisfy Condition \textup{(S)}, at least two of the eigenvectors
have disjoint support. Up to reordering, we may assume that 
$\operatorname{supp}\phi_1\cap\operatorname{supp}\phi_2=\emptyset$.

Now, let $u_0=\phi_1+\phi_2$ and, for $c=(c_1,c_2)\in \T^2$, write $u_c=c_1\phi_1+c_2\phi_2$. Then
$$
e^{-itH_Q}u_c(j)=\begin{cases}
    c_1e^{-i\lambda_1 t}\phi_1(j)&\mbox{if }j\in\operatorname{supp}\phi_1\\
    c_2e^{-i\lambda_2 t}\phi_2(j)&\mbox{if }j\in\operatorname{supp}\phi_2\\
    0&\mbox{otherwise}
\end{cases}
$$
so that $|e^{-itH_Q}u_c(j)|$ is independent of $c$. It follows that if $c_1\not=c_2$
then $u_c$ is an $\partner$-partner of $u_0$ that is not a trivial partner.
\end{proof}

\begin{example}[A graph that satisfies $B_2$,  \textup{(I)},  but not  \textup{(S)}]
Let us now construct a connected graph $G$ and a potential $Q$
that satisfies Conditions $B_2$,  \textup{(I)},  but not  \textup{(S)} so that $H_Q$ fails to do phase retrieval.

Consider the connected graph \(G\) on the vertex set $V=\{1,\ldots,6\}$
with edge set

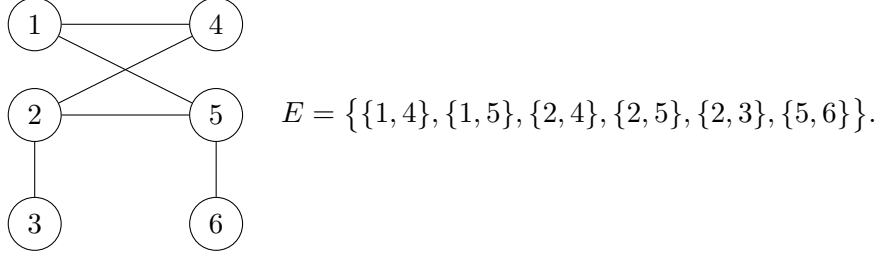
\begin{center}
\begin{tikzpicture}[scale=1.2]
  \node[circle, draw, minimum size=7mm] (1) at (0,1) {$1$};
  \node[circle, draw, minimum size=7mm] (2) at (0,0) {$2$};
  \node[circle, draw, minimum size=7mm] (3) at (0,-1.2) {$3$};
  \node[circle, draw, minimum size=7mm] (4) at (2,1) {$4$};
  \node[circle, draw, minimum size=7mm] (5) at (2,0) {$5$};
  \node[circle, draw, minimum size=7mm] (6) at (2,-1.2) {$6$};

\node at (6,0){$E=
\bigl\{
\{1,4\},\{1,5\},
\{2,4\},\{2,5\},
\{2,3\},\{5,6\}
\bigr\}.
$};

  \draw (1) -- (4);
  \draw (1) -- (5);
  \draw (2) -- (4);
  \draw (2) -- (5);
  \draw (2) -- (3);
  \draw (5) -- (6);
\end{tikzpicture}
\end{center}

Thus the vertices \(\{1,2\}\) and \(\{4,5\}\) span a copy of
\(K_{2,2}\), while vertex \(3\) is attached to \(2\) and vertex \(6\)
is attached to \(5\).
The degree matrix is thus $D=\operatorname{diag}(2,3,1,2,3,1)$.  

We choose the potential
$Q=\operatorname{diag}\left(1,3,\dfrac{7}{3},-3,-1,-\dfrac{5}{3}\right)$. 
Then
\begin{equation}
    \label{eq:66examplenoS}
H_Q=\Delta_G+Q
=
\begin{pmatrix}
3&0&0&-1&-1&0\\
0&6&-1&-1&-1&0\\
0&-1&\frac{10}{3}&0&0&0\\
-1&-1&0&-1&0&0\\
-1&-1&0&0&2&-1\\
0&0&0&0&-1&-\frac23
\end{pmatrix}.
\end{equation}

The eigenvalues of $H_Q$, in increasing order, are approximately
\[
\lambda_1\approx-1.43804889,
\qquad
\lambda_2=-1,
\qquad
\lambda_3\approx1.68235603,
\]
\[
\lambda_4=3,
\qquad
\lambda_5\approx3.73733104,
\qquad
\lambda_6\approx6.68502848.
\]
A direct computation shows that the \(21\) sums
$\lambda_j+\lambda_k$, $1\leq j\leq k\leq6$ are pairwise distinct, hence \(H_Q\) satisfies Condition \(B_2\).

The eigenvectors corresponding to $\lambda_2$ and $\lambda_4$ are given by
$$
\phi_2=(0,0,0,1,-1,-3)\quad\mbox{and}\quad\phi_4=(1,-1,-3,0,0,0).
$$
As they have disjoint support, Condition \((S)\) fails.

Finally, we verify Condition \((I)\) as reformulated in Section \ref{sec:reform}. Define
\[
B=
\left(
(H_Q^{m-1})_{jj}
\right)_{1\leq j,m\leq6}
=
\begin{pmatrix}
1&3&11&40&155&\frac{1873}{3}\\
1&6&39&\frac{769}{3}&\frac{15319}{9}&\frac{306199}{27}\\
1&\frac{10}{3}&\frac{109}{9}&\frac{1342}{27}
&\frac{19099}{81}&\frac{310069}{243}\\
1&-1&3&4&43&\frac{733}{3}\\
1&2&7&\frac{85}{3}&\frac{1303}{9}&\frac{22243}{27}\\
1&-\frac23&\frac{13}{9}&\frac{10}{27}
&\frac{475}{81}&\frac{4945}{243}
\end{pmatrix},
\]
so that $\det B=\dfrac{140707840}{81}\neq0$. As the eigenvalues are simple, $H_Q$ satisfies Condition  \textup{(I)}.
\end{example}

This example can be extended to obtain graphs of arbitrary size in the following way:
first we add new vertices $\{7,\ldots,N\}$ and for each new vertex $j$, we add two
edges $\{1,j\}$ and $\{2,j\}$. Then we consider a potential $Q_t$ such that
$$
H_{Q_t}=\begin{pmatrix}
    H_6&-C\\
    -C^T&D_t
\end{pmatrix}
$$
where $H_6$ is the $6\times 6$ matrix \eqref{eq:66examplenoS} from the previous example, $C$ is the $6\times (N-6)$
matrix with the first 2 rows containing only $1$'s and the other rows being $0$
and $D_t=\operatorname{diag}(t,3t,\ldots, 3^{N-7}t)$. One can prove  
that, for $t$ large enough, $H_{Q_t}$ still satisfies $B_2$ and  \textup{(I)} and that
$(0,0,0,1,-1,-3,0,\ldots,0)$ and
$(1,-1,-3,0,0,0,0,\ldots,0)$ are eigenvectors.
As their supports are disjoint,
Condition \textup{(S)} is not satisfied, so that $H_{Q_t}$ fails to do phase retrieval.

Such an example requires the graph to have at least 6 vertices:

\begin{proposition}\label{prop:b2andsimpliesS}
Let \(G=(V,E)\) be a connected graph with \(|V|\leq 5\), let
\(Q\in\mathbb{R}^{V}\), and set
\[
H_Q=\Delta_G+Q.
\]
If \(H_Q\) has simple eigenvalues and satisfies Condition \((I)\), then
it also satisfies Condition \((S)\).
\end{proposition}

\begin{proof} There is nothing to do for $|V|=1$.
Assume, by contradiction, that Condition \((S)\) does not hold. Then
there exist two eigenvectors \(\varphi\) and \(\psi\) of \(H_Q\) whose
supports are disjoint:
\[
\operatorname{supp}\varphi\cap\operatorname{supp}\psi=\varnothing.
\]

We first observe that a nonzero eigenvector of \(H_Q\) cannot be
supported at a single vertex. Indeed, suppose that
\(\operatorname{supp}\varphi=\{u\}\). Since \(G\) is connected and has
at least two vertices, there exists a vertex \(w\sim u\). The
eigenvalue equation at \(w\) gives
\[
0=\lambda\varphi(w)=(H_Q\varphi)(w)=-\varphi(u),
\]
which is impossible.

Consequently,
\[
|\operatorname{supp}\varphi|\geq 2,
\qquad
|\operatorname{supp}\psi|\geq 2.
\]
When $|V|=2$, both eigenvectors have then full support, 
and when $|V|=3$, the two supports have at least one common element. In both cases
Condition  \textup{(S)} is then satisfied. It therefore remains to consider the cases   $|V|=4$ or $5$.
Since the two supports are disjoint and \(|V|\leq 5\), at least one of
them has cardinality exactly \(2\). 

Without loss of generality, write
\[
\operatorname{supp}\varphi=\{u,v\}.
\]
Set
$\varphi(u)=a$ and $\varphi(v)=b$, 
where \(a,b\neq 0\). 
For every \(w\notin\{u,v\}\), the eigenvalue equation at \(w\) yields
\[
0=(H_Q-\lambda I)\varphi(w)
=
-\mathbf 1_{\{w\sim u\}}a
-\mathbf 1_{\{w\sim v\}}b.
\]
Because \(G\) is connected, there exists at least one vertex outside
\(\{u,v\}\) adjacent to \(u\) or \(v\). The preceding identity then
implies
\[
b=-a.
\]
Moreover, for every \(w\notin\{u,v\}\), $w\sim u$ if and only if $w\sim v$.
Thus \(u\) and \(v\) have the same neighbors outside \(\{u,v\}\).

Let
\[
h_{xx}=\deg(x)+Q(x)
\]
denote the diagonal entries of \(H_Q\), and let
\[
\varepsilon=
\begin{cases}
1,& u\sim v,\\
0,& u\not\sim v.
\end{cases}
\]
If \(H_Q\varphi=\lambda\varphi\), then the eigenvalue equations at
\(u\) and \(v\), together with \(b=-a\), give
\[
(h_{uu}+\varepsilon)a=\lambda a
\quad\mbox{and}\quad
(h_{vv}+\varepsilon)a=\lambda a.
\]
Since \(a,b\neq0\), it follows that $h_{uu}=h_{vv}$.

Therefore \(H_Q\) is invariant under the transposition exchanging the
vertices \(u\) and \(v\). More precisely, if \(P\) denotes the
permutation matrix associated with this transposition, then
\[
PH_Q=H_QP.
\]

The spectrum of
\(H_Q\) is simple. Hence every eigenvector \(\phi_k\) of \(H_Q\) is
also an eigenvector of \(P\). Since \(P^2=I\), the corresponding
eigenvalue of \(P\) is either \(1\) or \(-1\). Thus
\[
\phi_k(u)=\phi_k(v)
\qquad\text{or}\qquad
\phi_k(u)=-\phi_k(v),
\]
and in either case $\phi_k(u)^2=\phi_k(v)^2$.

It follows that the \(u\)-th and \(v\)-th rows of the
squared eigenvector matrix
\[
M=\bigl(\phi_k(j)^2\bigr)_{j,k\in V}
\]
are equal. Hence \(M\) is singular, contradicting Condition \((I)\).

Therefore,  Condition \((S)\) must hold.
\end{proof}

\subsection{Eigenvectors with the same pointwise modulus}
\label{sec:obs3}

We record another elementary obstruction, namely the simplest failure of Condition $(I)$
when two of the  squared-eigenvectors $\bigl(\phi_p^2(m)\bigr)_{m=1}^n$
and $\bigl(\phi_q^2(m)\bigr)_{m=1}^n$ are already linearly dependent.
As the $\phi_j$'s have norm one, this means that $\phi_p^2(m)=\phi_q^2(m)$.
The next lemma shows
that these two modes may be interchanged without changing the intensities.

\begin{lemma}[Two-mode swapping]
\label{lem:swap-ambiguity}
Let \(H\) be a real symmetric matrix on \(\C^n\), with real orthonormal
eigenbasis \(\{\phi_j\}_{j=1}^n\). 
Assume that for some \(p\neq q\),
\[
        \phi_p(m)^2=\phi_q(m)^2,
        \qquad m=1,\ldots,n.
\]
For \(\alpha,\beta\in\R\), set
\[
        u_0=\alpha\phi_p+\beta\phi_q,
        \qquad
        v_0=\beta\phi_p+\alpha\phi_q .
\]
Then
\[
        |(e^{-itH}u_0)(m)|
        =
        |(e^{-itH}v_0)(m)|,
        \qquad t\in\R,\ m=1,\ldots,n.
\]
Moreover, if \(\alpha\beta\neq0\) and \(\alpha\neq\pm\beta\), then
\[
        v_0\not\trivial u_0 .
\]
\end{lemma}

\begin{proof}
For each \(m\), write
\[
        \phi_q(m)=s_m\phi_p(m),
        \qquad s_m\in\{\pm1\},
\]
with \(s_m\) arbitrary when both entries vanish. Put
\[
        A=e^{-i\lambda_p t},
        \qquad
        B=e^{-i\lambda_q t}.
\]
Then
\[
        (e^{-itH}u_0)(m)
        =
        \phi_p(m)(\alpha A+s_m\beta B),
\]
and
\[
        (e^{-itH}v_0)(m)
        =
        \phi_p(m)(\beta A+s_m\alpha B).
\]
Since \(\alpha,\beta\in\R\),
\[
        |\alpha A+s_m\beta B|^2
        =
        \alpha^2+\beta^2
        +
        2s_m\alpha\beta\,\Re(A\overline B),
\]
and the right-hand side is unchanged after interchanging \(\alpha\) and
\(\beta\). Hence
\[
        |(e^{-itH}u_0)(m)|
        =
        |(e^{-itH}v_0)(m)|.
\]

It remains to check nontriviality. If \(v_0=cu_0\) for some \(c\in\mathbb T\),
then, taking inner products with \(\phi_p\) and \(\phi_q\),
\[
        \beta=c\alpha,
        \qquad
        \alpha=c\beta .
\]
Since \(\alpha\beta\neq0\), this gives \(c^2=1\), hence \(c=\pm1\), and
therefore \(\beta=\pm\alpha\). This contradicts the assumption. Thus
\(v_0\not\trivial u_0\).
\end{proof}

\begin{example}    
We have already seen in Example \ref{ex:2vertex} that this may happen for the two-vertex graph
$P_2=(\{1,2\},\{(1,2)\})$ when $Q=0$ --- or more generally
when the potential is constant $Q=(q_1,q_1)$. In this case, the two eigenvectors are given by
\[
        \phi^+=\frac1{\sqrt2}(1,1),
        \qquad
        \phi^-=\frac1{\sqrt2}(1,-1).
\]

Now let $G=(V,E)$ be another finite graph, $Q\in\R^V$ and $H_Q=\Delta_G+Q$. Extend $Q$ to $\{1,2\}\times V$
by setting $R(j,j')=Q(j')$ and take $H_R=\Delta_{P_2\square G}+R$.
Then if $\phi$ is an eigenvector of $H_Q$, $\phi^\pm\otimes\phi$ is an eigenvector of $H_R$.
As
$$
\bigl(\phi^\pm\otimes\phi(j,k)\bigr)^2=\phi^\pm(j)^2\phi(k)^2=\dfrac{1}{2}\phi(k)^2
$$
so that Lemma \ref{lem:swap-ambiguity} applies and $\Delta_{P_2\square G}$ fails to do phase retrieval.
\end{example}

\section{Conclusion and future work}\label{sec:conclusion}

In this paper, we investigated the phase retrieval problem for Schr\"odinger evolution on finite connected simple graphs,
focusing on Schr\"odinger operators that do phase retrieval.
We have highlighted the role of the potential for this to be the case.
In particular, we have constructed explicit potentials $Q$ such that $H_Q=\Delta_G+Q$ does Schr\"odinger phase
retrieval. We have also shown that for a generic potential, the same property holds.

The free case $Q=0$ is more rigid. For instance, we have shown that Condition \textup{(I)} is never
satisfied in this case. Further, it is easy to see that Condition $B_2$ does not hold for
any connected graph on four vertices. This alone does not mean that $\Delta_G$
fails to do phase retrieval but makes the task more difficult, especially when,
on top of $(I)$, one of the two other conditions
$B_2$ or $(S)$ is missing. It would therefore be interesting
to characterize graphs for which Conditions $B_2$ and \textup{(S)} hold.

It is known, for instance, \cite[Theorem 2.7]{ChristoffersenLuhNguyenWang},  that with high probability,
Erd\"os-Renyi graphs satisfy Condition $(S)$
(\cite{ChristoffersenLuhNguyenWang} shows that 
with high probability every Laplacian eigenvector has at most one coordinate smaller than $n^{-B}$) and that the eigenvalues are simple
but the stronger Condition $B_2$ is unknown to our knowledge.
Another question is to find
conditions on $G$ under which the solution of the Schr\"odinger phase retrieval
problem for $\Delta_G$ is unique up to a global phase for almost every initial condition $u_0$. We have seen in Theorem \ref{thm:almostall} that this is the case under Conditions $B_2$ and $(S)$.
We leave a further investigation to future work. 

Finally, another important direction that remains open is the development of reconstruction algorithms of the problem.
There are several families of algorithms that are adapted for finite-dimensional phase retrieval problems
from frame coefficients, such as Gerchberg--Saxton, PhaseLift, PhaseCut; {\it see,  e.g.},  \cite{BauschkeCombettesLuke2002,CandesEldarStrohmerVoroninski2013,CandesStrohmerVoroninski2013,WaldspurgerDAspremontMallat2015}. An interesting direction for future research would be to evaluate these algorithms on the Schrödinger phase retrieval problem.




%


\begin{thebibliography}{10}
\bibitem{AKT}
{\sc A. Aldroubi, I. Krishtal, and S. Tang},
{\em Phaseless reconstruction from space–time samples}, Appl.
Comput. Harmon. Anal., {\bf 48} (2020), 395--414.

\bibitem{AllainAslanCoeneEtAl2025}
{\sc M.~Allain, S.~Aslan, W.~Coene, S.~Dirksen, J.~Dong, J.~Flamant, M.~Iwen,
  F.~Krahmer, T.~van Leeuwen, O.~Melnyk, A.~Menzel, A.~P. Mosk, V.~Nikitin,
  P.~Salanevich, G.~Plonka, and M.~Wellershoff}, {\em Phasebook: a survey of
  selected open problems in phase retrieval}, Sampl. Theory Signal Process.
  Data Anal., {\bf 23} (2025) Article 23.

\bibitem{BalanBodmannCasazzaEdidin2009}
{\sc R.~Balan, B.~G. Bodmann, P.~G. Casazza, and D.~Edidin}, {\em Painless
  reconstruction from magnitudes of frame coefficients}, J. Fourier Anal.
  Appl., {\bf 15} (2009), 488--501.

\bibitem{BCE}
{\sc R.~Balan, P.~Casazza, and D.~Edidin}, {\em On signal reconstruction
  without phase}, Appl. Comput. Harmon. Anal., {\bf 20} (2006), 345--356.

\bibitem{BH}
{\sc R. Beinert and M. Hasannasab}, 
{\em Phase retrieval and system identification in dynamical sampling via Prony’s method},
Adv. Comput. Math. {\bf 49} (2023), Article 56.

\bibitem{BauschkeCombettesLuke2002}
{\sc H.~H. Bauschke, P.~L. Combettes, and D.~R. Luke}, {\em Phase retrieval,
  error reduction algorithm, and {Fienup} variants: a view from convex
  optimization}, J. Opt. Soc. Am. A, {\bf 19} (2002),
  1334--1345.

\bibitem{BelousovIsmagilov2008Pauli}
{\sc P.~A. Belousov and R.~S. Ismagilov}, {\em {Pauli} problem and related
  mathematical problems}, Theoret. Math. Phys., {\bf 157} (2008), 1365--1369.

\bibitem{BF}
{\sc I. Bojarovska and A. Flinth}, 
{\em Phase Retrieval from Gabor Measurements,} J. Fourier Anal. Appl. {\bf 22} (2016), 542-567. 

\bibitem{brouwer2011spectra}
{\sc A.~Brouwer and W.~Haemers}, {\em Spectra of Graphs}, Springer Science \&
  Business Media, 2011.

\bibitem{CandesEldarStrohmerVoroninski2013}
{\sc E.~J. Cand{\`e}s, Y.~C. Eldar, T.~Strohmer, and V.~Voroninski}, {\em Phase
  retrieval via matrix completion}, SIAM J. Imaging Sci., {\bf 6} (2013),
199--225.

\bibitem{CandesStrohmerVoroninski2013}
{\sc E.~J. Cand{\`e}s, T.~Strohmer, and V.~Voroninski}, {\em {PhaseLift}: exact
  and stable signal recovery from magnitude measurements via convex
  programming}, Comm. Pure Appl. Math., {\bf 66} (2013), 1241--1274.

\bibitem{ChristoffersenLuhNguyenWang}
{\sc N.~Christoffersen, K.~Luh, H.~H.~Nguyen and J.~Wang},
{\em Eigenvalue gaps of the Laplacian of random graphs},
arXiv:2501.00234, 2025.

\bibitem{Co}
{\sc J.~V. Corbett}, {\em The {Pauli} problem, state reconstruction and
  quantum-real numbers}, Rep. Math. Phys., {\bf 57} (2006), 53--68.

\bibitem{CH}
{\sc J.~V. Corbett and C.~A. Hurst}, {\em Are wave functions uniquely
  determined by their position and momentum distributions?}, J.
  Austr. Math. Soc. B, {\bf 20} (1978), 182--201.

\bibitem{Fienup1982}
{\sc J.~R. Fienup}, {\em Phase retrieval algorithms: a comparison}, Appl. Opt.,
  {\bf 21} (1982), 2758--2769.

\bibitem{Godsil2012ControllableSubsets}
{\sc C.~Godsil}, {\em Controllable subsets in graphs}, Ann. Comb.,
  {\bf 16} (2012), 733--744.

\bibitem{Go}
{\sc D.~Goyeneche, G.~Ca{\~n}as, S.~Etcheverry, E.~S. G{\'o}mez, G.~B. Xavier,
  G.~Lima, and A.~Delgado}, {\em Five measurement bases determine pure quantum
  states on any dimension}, Phys. Rev. Lett., {\bf 115} (2015), ~090401.

\bibitem{GKK}
{\sc D. Gross, F. Krahmer, and R. Kueng},
{\em Improved recovery guarantees for phase retrieval from coded diffraction patterns},
Appl. Comput. Harmon. Anal. {\bf 42} (2017) 37–-64.

\bibitem{GrohsKoppensteinerRathmair2020}
{\sc P.~Grohs, S.~Koppensteiner, and M.~Rathmair}, {\em Phase retrieval:
  uniqueness and stability}, SIAM Rev., {\bf 62} (2020), 301--350.

\bibitem{MonfaredShader2016NowhereZero}
{\sc K.~Hassani~Monfared and B.~L. Shader}, {\em The nowhere-zero eigenbasis
  problem for a graph}, Linear Algebra Appl., {\bf 505} (2016), 296--312.

\bibitem{HMW}
{\sc T.~Heinosaari, L.~Mazzarella, and M.~M. Wolf}, {\em Quantum tomography
  under prior information}, Comm. Math. Phys., {\bf 318} (2013), 355--374.

\bibitem{HornJohnson}
{\sc R.~A. Horn and C.~R. Johnson},
\emph{Matrix Analysis}, 2nd ed., Cambridge University Press, Cambridge, 2013.

\bibitem{Is}
{\sc R.~S. Ismagilov}, {\em On the {Pauli} problem}, Funktsional. Anal. i
 Prilozhen., {\bf 30} (1996), 82--84.

\bibitem{Jaming1999Radar}
{\sc P.~Jaming}, {\em Phase retrieval techniques for radar ambiguity problems},
  J. Fourier Anal. Appl., {\bf 5} (1999), 309--329.

\bibitem{JamingAcha}
{\sc P.~Jaming}, {\em Uniqueness results in an extension of {Pauli}'s phase
  retrieval}, Appl. Comput. Harmon. Anal., {\bf 37} (2014), 413--441.

\bibitem{Janssen1992Zak}
{\sc A.~J. E.~M. Janssen}, {\em The {Zak} transform and some counterexamples in
  time-frequency analysis}, IEEE Trans. Inform. Theory, {\bf 38} (1992),
  168--171.

\bibitem{KlibanovSacksTikhonravov1995}
{\sc M.~V. Klibanov, P.~E. Sacks, and A.~V. Tikhonravov}, {\em The phase
  retrieval problem}, Inverse Prob., {\bf 11} (1995), 1--28.

\bibitem{LeveneOblakSmigoc2024GenericEigenvectors}
{\sc R.~H. Levene, P.~Oblak, and H.~{\v{S}}migoc}, {\em Distinct eigenvalues
  are realizable with generic eigenvectors}, Linear and Multilinear Algebra, {\bf 72}
  (2024), 2054--2068.

\bibitem{Mixon2015PhaseTransitions}
{\sc D.~G. Mixon}, {\em Phase transitions in phase retrieval}, in Excursions in
  Harmonic Analysis, Volume 4, R.~Balan, M.~Begu{\'e}, J.~J. Benedetto,
  W.~Czaja, and K.~A. Okoudjou, eds., Applied and Numerical Harmonic Analysis,
  Birkh{\"a}user, 2015, 123--147.

\bibitem{MV}
{\sc D.~Mondragon and V.~Voroninski}, {\em Determination of all pure quantum
  states from a minimal number of observables}. {\tt arXiv:1306.1214}

\bibitem{MP}
{\sc B.~Z. Moroz and A.~M. Perelomov}, {\em On a problem posed by {Pauli}},
  Theoret. Math. Phys., {\bf 101} (1994), 1200--1204.

\bibitem{Pa}
{\sc W.~Pauli}, {\em Die allgemeinen prinzipien der wellenmechanik}, in
  Handbuch Der Physik, H.~Geiger and K.~Scheel, eds., vol.~24, Springer-Verlag,
  Berlin, 1933, 83--272.
\newblock English translation: \emph{General Principles of Quantum Mechanics},
  Springer-Verlag, Berlin, 1980.

\bibitem{PoignardPereiraPade2018LaplacianSpectra}
{\sc C.~Poignard, T.~Pereira, and J.~P. Pade}, {\em Spectra of {Laplacian}
  matrices of weighted graphs: structural genericity properties}, SIAM J. Appl.
  Math., {\bf 78} (2018), 372--394.

\bibitem{Re}
{\sc H.~Reichenbach}, {\em Philosophic Foundations of Quantum Mechanics},
  University of California Press, Berkeley, 1944.

\bibitem{ShechtmanEldarCohenEtAl2015}
{\sc Y.~Shechtman, Y.~C. Eldar, O.~Cohen, H.~N. Chapman, J.~Miao, and
  M.~Segev}, {\em Phase retrieval with application to optical imaging: a
  contemporary overview}, IEEE Signal Process. Mag., {\bf 32} (2015), 87--109.

\bibitem{WaldspurgerDAspremontMallat2015}
{\sc I.~Waldspurger, A.~d'Aspremont, and S.~Mallat}, {\em Phase recovery,
  {MaxCut} and complex semidefinite programming}, Math. Program., {\bf 149} (2015),
  47--81.
\end{thebibliography}
\end{document}